\documentclass{amsart}
\pdfoutput=1
\usepackage{amsfonts,amsmath,amscd,amssymb,amsthm,graphicx}
\usepackage{hyperref}
\usepackage[mathscr]{euscript} 
\usepackage{wrapfig}
\usepackage{fixltx2e}

\usepackage{algorithm}
\usepackage{color}
\usepackage[noend]{algpseudocode}

\algnewcommand\To{\textbf{to}}
\renewcommand{\Comment}[2][.5\linewidth]{%
  \leavevmode\hfill\makebox[#1][l]{//~#2}}

\usepackage{mathtools}

\newtheorem{theorem}{Theorem}
\newtheorem{definition}[theorem]{Definition}
\newtheorem{lemma}[theorem]{Lemma}
\newtheorem{assumption}{Assumption}

\usepackage{cleveref}
\crefformat{equation}{(#2#1#3)}
\Crefname{ALC@unique}{Line}{Lines}
\crefname{assumption}{Assumption}{Assumptions}

\renewcommand{\b}[1]{{\boldsymbol{#1}}}
\renewcommand{\c}[1]{{\mathcal{#1}}}
\newcommand{\bc}[1]{{\b{\c{#1}}}}
\newcommand{\RR}{\mathbb{R}}
\newcommand{\PP}{\mathbb{P}}
\renewcommand{\SS}{\mathbb{S}}
\newcommand{\NN}{\mathbb{N}}

\newcommand{\MG}{M\!G}

\DeclareMathOperator{\diag}{diag}
\DeclareMathOperator{\blockdiag}{blockdiag}
\DeclareMathOperator{\Exp}{\mathbb{E}}
\DeclareMathOperator{\Var}{\mathbb{V}}
\DeclareMathOperator{\Cov}{\mathrm{Cov}}
\providecommand{\norm}[2]{\left\lVert#1\right\rVert_{#2}}

\usepackage[makeroom,thicklines]{cancel}

\newcommand{\pc}[1]{\parbox[c]{1cm}{\includegraphics[width=1cm]{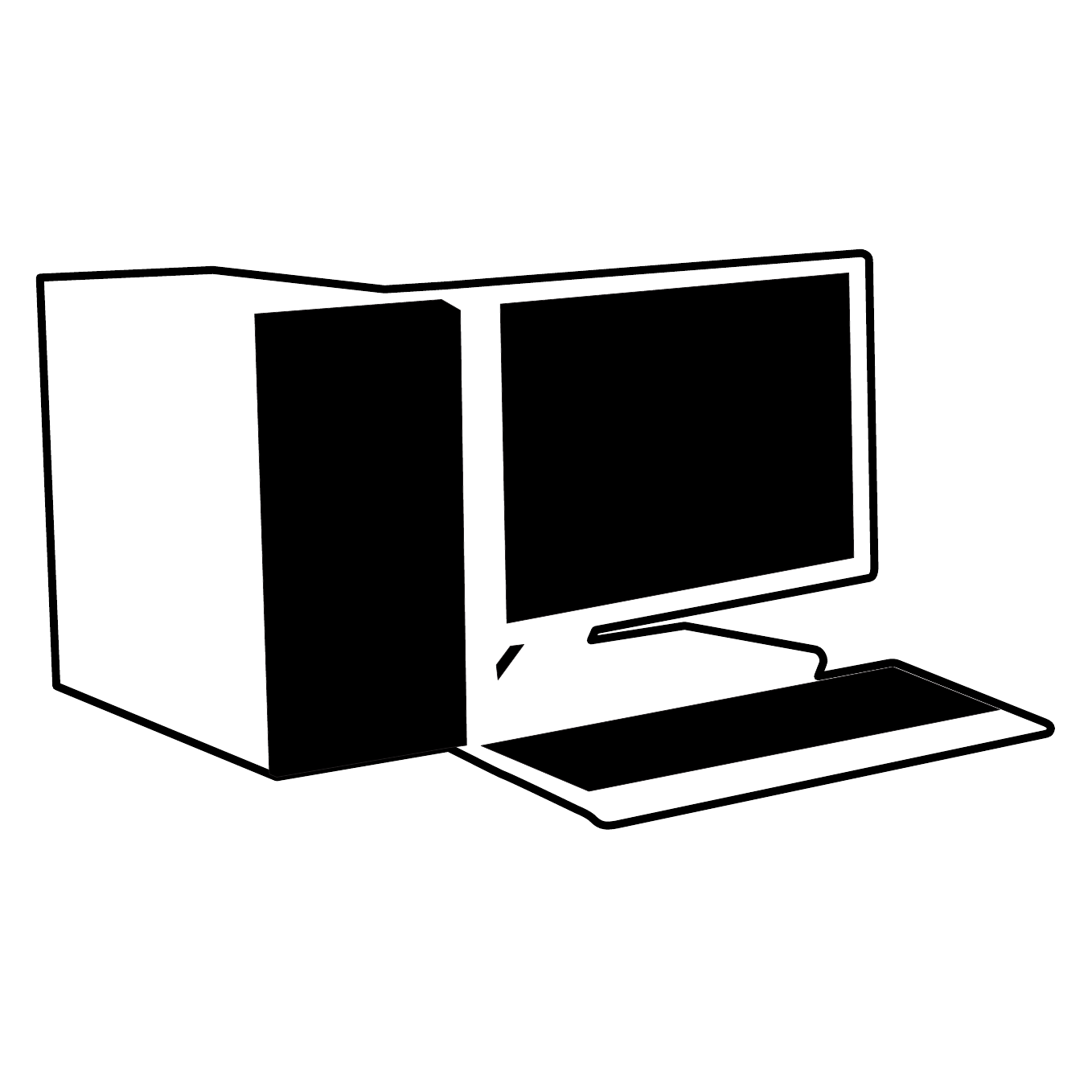}}\textsubscript{#1}}

\usepackage{cite}

\usepackage[super]{nth}

\usepackage{thmtools,thm-restate}

\allowdisplaybreaks

\title[Is the multigrid method fault tolerant? The Two-Grid Case.]{Is the multigrid method fault tolerant? \\ The Two-Grid Case.}
\thanks{The effort of MA was partially supported by SIRIUS award from DoE.}

\author{Mark Ainsworth}
\address{Division of Applied Mathematics, Brown University, 182 George St, Providence, RI 02912, USA.\newline
  Computer Science and Mathematics Division, Oak Ridge National Laboratory, Oak Ridge, TN 37831, USA.}
\email{Mark\_Ainsworth@Brown.edu}
\author{Christian Glusa}
\address{Division of Applied Mathematics, Brown University, 182 George St, Providence, RI 02912, USA.}
\email{Christian\_Glusa@Brown.edu}

\begin{document}

\dedicatory{This paper is dedicated to Professor Ivo Babu\v{s}ka on the occasion of his \nth{90} birthday}

\begin{abstract}
  The predicted reduced resiliency of next-generation high performance
  computers means that it will become necessary to take into account
  the effects of randomly occurring faults on numerical methods.
  Further, in the event of a hard fault occurring, a decision has to
  be made as to what remedial action should be taken in order to
  resume the execution of the algorithm. The action that is chosen can
  have a dramatic effect on the performance and characteristics of the
  scheme. Ideally, the resulting algorithm should be subjected to the
  same kind of mathematical analysis that was applied to the original,
  deterministic variant.

  The purpose of this work is to provide an analysis of the behaviour
  of the multigrid algorithm in the presence of faults. Multigrid is
  arguably the method of choice for the solution of large\--scale linear
  algebra problems arising from discretization of partial differential
  equations and it is of considerable importance to anticipate its
  behaviour on an exascale machine. The analysis of resilience of
  algorithms is in its infancy and the current work is perhaps the
  first to provide a mathematical model for faults and analyse the
  behaviour of a state-of-the-art algorithm under the model. It is
  shown that the Two Grid Method fails to be resilient to faults.
  Attention is then turned to identifying the minimal necessary
  remedial action required to restore the rate of convergence to that
  enjoyed by the ideal fault-free method.
\end{abstract}


\maketitle

\section{Introduction}

President Obama's executive order in the summer of 2015 establishing
the National Strategic Computing Initiative\footnote{Creating a
  National Strategic Computing Initiative, Executive Order No. 13702
  of July 29th, 2015, Federal Register} committed the US to the
development of a \emph{capable} exascale computing system. Given that
the performance of the current number one machine Tianhe-2 is roughly
one thirtieth of that of an exascale system, it is easy to
underestimate the challenge posed by this task. One way to envisage
the scale of the undertaking is that the combined processing power of
the \emph{entire} TOP 500 list is less than half of one exaflop
(\(10^{18}\) floating point operations per second).

It is widely accepted that an exascale machine should respect a 20MW
power envelope. Tianhe-2 already consumes 18MW of power, and if it
were possible to simply upscale to exascale using the current
technology, would require around 540MW or roughly the same amount of
energy required to power half a million homes. In order to meet the
power envelope, and other requirements, one is required to change the
physical operating limits of the machine through such mechanisms as
lower voltage logic thresholds, reduced cell capacitance and through
minimisation of data movement. All of these of factors contribute to a
lower overall reliability of the machine in terms of random
bit-flipping and corruption of logic states from extraneous sources
such as cosmic rays and, failures of individual components. The issue
is further exacerbated by the additional numbers of hardware
components required for an exascale system
\cite{CappelloGeistEtAl2009_TowardExascaleResilience,
  CappelloGeistEtAl2014_TowardExascaleResilience,
  Cappello2009_FaultTolerancePetascaleExascaleSystems,
  DongarraHittingerEtAl2014_AppliedMathematicsResearchExascaleComputing}.

The taxonomy of faults and failures by Avi\v{z}ienis et
al.~\cite{AvizienisLaprieEtAl2004_BasicConceptsTaxonomyDependableSecureComputing}
distinguishes between \emph{hard faults} and \emph{soft faults}. Soft
faults correspond to corruption of instructions as well as the data
produced and used within the code but which allow the execution of the
code to proceed albeit in a corrupted state. An example of such a
fault would be flipping of individual bits in a floating point number.
Some bit flips in the exponent, sign or in the significant components
of the mantissa may result in a relatively large error that becomes
detectable. Hard faults might correspond to a message from one or more
compute nodes either being corrupted beyond all recognition (possibly
resulting in an exception being thrown), or lost altogether in the
event of a compute node failing completely. Hard faults result in the
interruption of the execution of the actual code unless remedial
action is taken. As such, hard faults constitute errors that are
readily \emph{detectable}. Although the likelihood $q$ of a hard fault
occurring can be assumed small, it is not negligible. The fact that an
algorithm cannot continue after a hard fault without some kind of
remedial action means that such faults cannot simply be ignored in the
hope that the algorithm will recover naturally.

Most, if not all, existing algorithms were derived under the
assumption that such faults cannot occur and accordingly, no
indication is generally offered as to what might constitute an
appropriate course of action should a fault be detected in a
particular algorithm. Nevertheless, in the event of a hard fault
occurring, a decision has to be made as to what remedial action should
be taken in order to resume the execution of the algorithm. The action
that is chosen can have a dramatic effect on the performance and
characteristics of the algorithm and as such represents a vital part
of the algorithm. Ideally, the resulting algorithm should be subjected
to the same kind of mathematical analysis that was applied to the
original, deterministic, variant of the algorithm. Historical
considerations mean that such analysis is available for very few
methods.

At the present time, the basic approach to dealing with faults
consists of attempting to restore the algorithm to the state that
would have existed had no fault occurred. Two common ways in which
this might be accomplished include:
\emph{Checkpoint-restart}~\cite{HeraultRobert2015_FaultToleranceTechniquesHighPerformanceComputing,
  SnirWisniewskiEtAl2014_AddressingFailuresExascaleComputing} whereby
the state of the system is stored (or checkpointed) at pre-determined
intervals so that in the event of a fault occurring, the computation
can be restarted from the stored state; or, \emph{Process
  Replication}~\cite{HeraultRobert2015_FaultToleranceTechniquesHighPerformanceComputing,
  SnirWisniewskiEtAl2014_AddressingFailuresExascaleComputing}, whereby
each critical component of the overall process is replicated one or
more times so that in the event of a component being subject to a
fault, the true state can be restored by making reference to an
unaffected replica. Hybrid variants of these strategies can also be
contemplated. The drawbacks of such approaches are self-evident.

The purpose of this work is to provide an analysis of the behaviour of
the multigrid algorithm on a machine prone to faults. The multigrid
method is arguably the method of choice for the solution of
large-scale linear algebra problems arising from discretization of
partial differential equations. Consequently, it is of considerable
importance to anticipate its behaviour on an exascale machine. The
analysis of resilience of algorithms is in its infancy and the current
work is perhaps the first to provide a mathematical model for faults
and analyse the behaviour of a state-of-the-art algorithm under the
model.

This work is organised as follows. In the next section we present a
mathematical model for faults and a strategy for their mitigation. In
\Cref{sec:model-fault-prone} we will briefly introduce the multigrid
algorithm and its variant when faults are taken into account. The
random iteration matrix of the resulting \emph{Fault-Prone Multigrid
  Method} is determined and an appropriate metric for convergence is
discussed. The main results of this work concerning the rate of
convergence of the Fault-Prone Two Grid method is found in
\Cref{sec:summary-main-results}. \Cref{thm:perturbedTGconv}
demonstrates that the Two Grid Method is not fault resilient. The
minimal necessary remedial action is given by \Cref{thm:TGNoProlong}:
protecting the prolongation restores the rate of convergence of the
ideal fault-free method. Supporting numerical evidence is provided.

\subsection{Related Work}

Different techniques have been previously employed in order to achieve
fault resilience for iterative methods in general and multigrid in
particular.
Replication was used in
\cite{CuiXuEtAl2013_ErrorResilientRedundantSubspaceCorrectionMethod,
  CasasSupinskiEtAl2012_FaultResilienceAlgebraicMultiGridSolver}, and
checkpoint-restart in
\cite{CalhounOlsonEtAl2015_TowardsMoreFaultResilientMultigridSolver}.
Stoyanov and Webster
\cite{StoyanovWebster2015_NumericalAnalysisFixedPoint} proposed a
method based on selective reliability for fixed point methods.
Finally, Huber et al.
\cite{HuberGmeinerEtAl2015_ResilienceMultigridSoftwareAtExtremeScale}
proposed a recovery method to mitigate the effect of hard faults.

\section{Modelling Faults}

In order to take account of the effect of detectable faults on the
stability and convergence of a numerical algorithm, we develop a
simple probabilistic model for this behaviour and describe how it is
incorporated into the formulation of iterative algorithms. As a matter of
fact, it will transpire that our model for detectable faults will also
apply to the case of silent faults that are relatively small.

Our preferred approach to mitigating the effects of faults in
multigrid and related iterative algorithms is rather simple: \emph{the
  lost or corrupted result expected from a node is replaced by a
  zero}. That is to say, if a value $x\in\RR$ is prone to possible
hard faults, then we propose to continue the global computation using
$\tilde{x}\in\RR$ in place of $x$, where
\begin{align*}
  \tilde{x} & =
            \begin{cases}
              0          & \text{if a fault is detected}, \\
              x          & \text{otherwise}.
            \end{cases}
\end{align*}

This minimalist, \emph{laissez-faire}, approach has obvious
attractions in terms of restoring the system to a valid state without
the need to (i) halt the execution (other than to perhaps migrate data
to a substitute node in the event of a node failure); (ii) take or
read any data checkpoints; (iii) recompute any lost quantities; or,
(iv) compromise on resources through having to replicate processes.
However, the efficacy of \emph{laissez-faire} depends on the extent to
which the convergence characteristics of the resulting algorithm
mirror those of the fault-free deterministic variant. The subject of
the present work is to carry out a detailed and rigorous mathematical
analysis of this question in the context of the Two Grid Algorithm.
The multigrid case will be considered in our subsequent work.

\subsection{Probabilistic Model of Faults}

In order to model the effects of the laissez-faire fault mitigation approach on
an algorithm, we introduce a Bernoulli random variable given by
\begin{align}\label{BernoulliRV}
  \chi &=
       \begin{cases}
              0          & \text{with probability } q, \\
              1          & \text{with probability } 1-q.
       \end{cases}
\end{align}
If a scalar variable $x$ is subject to faults, then the effect of the fault and
the laissez-faire strategy is modelled by replacing the value of $x$ by the new
value $\tilde{x} = \chi x$.  Evidently $\tilde{x}$ is a random variable with mean
and variance given by $\Exp[\tilde{x}] = (1-q)x$ and $\Var[\tilde{x}]
=q(1-q)x^2$. By the same token, the effect of a fault on a vector-valued variable
$x\in\mathbb{R}^n$ is modelled in a similar fashion by defining $\tilde{x} =
\bc{X}x$,  where
\begin{align*}
  \bc{X}= \diag\left(\chi_{1},\dots,\chi_{n}\right)
\end{align*}
and $\chi_i$ are identically distributed Bernoulli random variables.
The variables \(\chi_{i}\) can be independent, thus modelling
componentwise faults, or block-dependent, as would be the case for a
node failure.

More formally, for given \(\varepsilon>0\), let $\SS_{\varepsilon}$ denote a set consisting of random matrices satisfying
the following conditions:
\begin{assumption} \label{as:faults}\leavevmode
  \begin{enumerate}

  \item Each $\bc{X}\in\SS_{\varepsilon}$ is a random diagonal matrix.

  \item For every $\bc{X}\in\SS_{\varepsilon}$, there holds   $\Exp[\bc{X}] =
    e(\bc{X})\b{I}$, where $e(\bc{X})>0$, and
    $|e(\bc{X})-1|\le C\varepsilon$ for some fixed $C>0$.

  \item For every
    $\bc{X}\in \SS_{\varepsilon}$ there holds $\norm{\Var[\bc{X}]}{2}
    = \max_{i,j}\left|\Cov\left(\bc{X}_{ii},\bc{X}_{jj}\right)
    \right| \leq\varepsilon$.

  \end{enumerate}
\end{assumption}

The parameter $\varepsilon$ is regarded as being small meaning that the random
matrices $\bc{X}$ behave, with high probability, like an identity matrix.  The
matrices $\bc{X}$ will appear at various points in our model for the
Fault-Prone Multigrid Algorithm. It is easy to see that if $\bc{X}_1$ and
$\bc{X}_2$ are independent diagonal matrices, then $\bc{X}_1\bc{X}_2$ is again
a random diagonal matrix.

Sometimes, a random matrix will appear in two (or more) different places in a
single equation and it will be important to clearly distinguish between the
cases where (i) each of the two matrices is a \emph{different} realisation of
the random matrix, or (ii) the matrices both correspond to the \emph{same}
realisation of the random matrix.  We shall adopt the convention whereby should
a symbol appear twice, then (ii) holds: i.e. the two occurrences represent the
\emph{same realisation} of the random matrix and the matrices are therefore
identical. However, if the square or higher power of a random matrix appears
then (i) holds: i.e. each of the matrices in the product is a \emph{different}
realisation of the same random matrix.

\subsection{Application to Silent Faults}

As mentioned above, a typical example of a soft fault would be
flipping of individual bits in a floating point number. Some bit flips
in the exponent, sign or in the significant components of the mantissa
may result in a relatively large error that becomes detectable. Such
cases could be treated using the laissez-faire strategy described
earlier. However, in other cases such faults may give a relatively
small error and, as a result, be difficult (or impossible) to
identify. Suppose that a vector $x$ is subject to such a silent fault
resulting in $x$ being replaced by $\tilde{x}$ in the machine. If the
relative error is at most $\varepsilon$, then
\begin{align}\label{Silent}
    \tilde{x}_i - x_i = \varepsilon_i\chi_i x_i
\end{align}
where $\chi_i$ is a Bernoulli random variable as in~\cref{BernoulliRV}, and
$\varepsilon_i$ is a random variable on $[-\varepsilon,\varepsilon]$.
Equally well, this means that
\begin{align*}
   \tilde{x} = \b{\Upsilon} x
\end{align*}
where $\b{\Upsilon} = \b{I}+\diag\left(\varepsilon_1\chi_1, \dots, \varepsilon_n\chi_n\right)$ is a random matrix that
satisfies the previous conditions required for membership of $\SS_{\varepsilon}$.

\subsection{Scope of Faults Covered by the Analysis}

Our analysis will cover faults that can be represented by a random
matrix belonging to $\SS_{\varepsilon}$. This means that the analysis
will cover each of the cases (i) when the fault is detectable and
mitigated using the laissez-faire strategy, and (ii) when the fault is
silent but is relatively small in the sense that it may be modelled
using~\eqref{Silent}. However, our analysis will not cover the
important case involving faults which result in entries in the
matrices or right hand sides in the problem being corrupted. Equally
well, our analysis will not cover the case of bit flips that result in
a large relative error but which nevertheless remain undetected.
Finally, this work is restricted to the solve phase of multigrid; we
assume that the setup phase is protected from faults.

\section{A Model for Fault-Prone Multigrid}
\label{sec:model-fault-prone}
\subsection{Multigrid Algorithm}
\label{sec:multigrid-algorithm}

Let $\b{A}$ be a symmetric, positive definite matrix arising from a
finite element discretization of an elliptic partial differential
equation in $d$ spatial dimensions.

We wish to use a multigrid method to compute the solution of the problem
\begin{align}\label{eq:Ax=b}
   \b{A}x = b
\end{align}
for a given load vector $b$. The multigrid method will utilise a nested
hierarchy of progressively coarser grids of dimension $0<n_0<n_1<\ldots<n_L$.
Restriction and prolongation operators are used to transfer vectors from one
level in the hierarchy to the next:
\begin{align*}
  \b{R}_\ell^{\ell+1}: \RR^{n_{\ell+1}}\rightarrow \RR^{n_\ell}, \quad
  \b{P}_{\ell+1}^\ell: \RR^{n_\ell}\rightarrow \RR^{n_{\ell+1}}.
\end{align*}
It will be assumed that, as is often the case, these operators are related by the rule
$\b{R}_\ell^{\ell+1} = \left(\b{P}_{\ell+1}^\ell\right)^{T}$.  A sequence of matrices
$\{\b{A}_\ell\}_{\ell=0}^L$ on the coarser levels is defined recursively as follows
\begin{align*}
	\b{A}_L = \b{A}; \quad
	\b{A}_\ell  = \b{R}\b{A}_{\ell+1}\b{P}, \quad \ell=0,\dots, L-1.
\end{align*}
Here and in what follows, the indices on the prolongation and restriction are
omitted whenever the appropriate choice of operator is clear from the context.

Smootheners are defined on each level $\ell=1,\ldots,L$ in the form
\begin{align}\label{eq:smoothener}
	S_\ell(b_\ell, x_\ell) = x_\ell + \b{N}_\ell (b_\ell - \b{A}_\ell x_\ell),
\end{align}
where $\b{N}_\ell$ is an approximation to the inverse of $\b{A}_\ell$. In particular,
choosing the matrix $\b{N}_\ell=\theta \b{D}_\ell^{-1}$ corresponds to the damped
Jacobi smoothener with damping factor \(\theta\).

The multigrid method is given in \Cref{alg:multigrid} and may be
invoked to obtain an approximate solution of problem~\cref{eq:Ax=b} by a call
to $\MG_L(b,0)$.
\begin{algorithm}
  \begin{algorithmic}[1]
    \Require Right hand side $b_\ell$; Initial iterate $x_\ell$
    \Ensure $\MG_\ell(b_\ell, x_\ell)$
    \If{$\ell=0$}
    \Return $\b{A}_{0}^{-1}b_{0}$  \Comment{Exact solve on coarsest grid}
    \Else
    \For{$i\leftarrow 1$ \To{} $\nu_{1}$}
    \State $x_\ell\leftarrow S_\ell\left(b_\ell,x_\ell\right)$  \Comment{$\nu_1$ pre-smoothing steps}
    \EndFor
    \State $d_{\ell-1}\leftarrow\b{R}\left(b_\ell-\b{A}_\ell x_\ell\right)$  \Comment{Restriction to coarser grid}
    \State $e_{\ell-1}\leftarrow 0$
    \For{$j\leftarrow 1$ \To{} $\gamma$}
    \State $e_{\ell-1}\leftarrow \MG_{\ell-1}(d_{\ell-1},e_{\ell-1})$  \Comment{$\gamma$ coarse grid correction steps}
    \EndFor
    \State $x_\ell\leftarrow x_\ell + \b{P}e_{\ell-1}$  \Comment{Prolongation to finer grid}
    \For{$i\leftarrow 1$ \To{} $\nu_{2}$}
    \State $x_\ell\leftarrow S_\ell\left(b_\ell,x_\ell\right)$  \Comment{$\nu_2$ post-smoothing steps}
    \EndFor
    \EndIf
  \end{algorithmic}
  \caption{Fault-free multigrid method $\MG_{\ell}$}
  \label{alg:multigrid}
\end{algorithm}


\subsection{A Model for Fault-Prone Multigrid}

The multigrid algorithm comprises a number of steps each of which may
be affected by faults were the algorithm to be implemented on a
fault-prone machine. In the absence of faults, a single iteration of
multigrid replaces the current iterate $x_\ell$ by
$\MG_\ell(b_\ell, x_\ell)$ defined in \Cref{alg:multigrid}. However,
in the presence of faults, a single iteration of multigrid means that
$x_\ell$ is replaced by $\c{\MG}_\ell(b_\ell, x_\ell)$ where
$\c{\MG}_\ell$ differs from $\MG_\ell$ in general due to corruption of
intermediate steps in the computations arising from faults. The object
of this section is to develop a model for faults and define the
corresponding operators $\c{\MG}_\ell$, $\ell\in\NN$. For simplicity,
we assume that the coarse grid is of moderate size meaning that the
exact solve $\b{A}_0^{-1}$ on the coarsest grid is not prone to
faults, i.e. $\c{\MG}_0(d_0,\cdot)=\b{A}_0^{-1}d_0$. This is a
reasonable assumption when the size of the coarse grid problem is
sufficiently small that either faults are not an issue or, if they
are, then replication can be used to achieve fault resilience. The
extension of the analysis to the case of fault-prone coarse grid
correction does not pose any fundamental difficulties but is not
pursued in the present work.

\subsection{Smoothening}
\label{sec:smoothening}

A single application of the smoothener $S_\ell$ to an iterate $x_\ell$
takes the form
\begin{align}
	S_\ell(b_\ell, x_\ell) = x_\ell + \b{N}_\ell (b_\ell - \b{A}_\ell x_\ell) \label{eq:1}
\end{align}
with each of the interior sub-steps in \cref{eq:1} being susceptible to faults.

The innermost step is the computation of the residual
$\rho_\ell = b_\ell - \b{A}_\ell x_\ell$. The action of the matrix
$\b{A}_\ell$ is applied repeatedly throughout the solution phase. We
shall assume that neither the entries or structure of matrix
\(\b{A}_{\ell}\) nor the right hand side \(b_{\ell}\) are subject to
corruption, only computation involving them. This would be the case
were the information needed to compute the action of $\b{A}_\ell$
placed in non-volatile random access memory (NVRAM) with relatively
modest overhead.
The net effect of faults in the computation of \(\rho_{\ell}\) is that
the preconditioner $\b{N}_\ell$ does not act on the true residual but
rather on a corrupted version modelled by $\bc{X}_1\rho_\ell$ where
$\bc{X}_1$ is a random diagonal matrix.

\begin{figure}
  \centering
  \begin{tabular*}{1.0\linewidth}{llrcl}
    Global: &\phantom{\xcancel{\pc{1}}}&\(x_{\ell}\)
    & \(\longrightarrow\)
    & \(x_{\ell}+\bc{X}_{\ell}^{(S)}\b{N}_{\ell}\left(b_{\ell}-\b{A}_{\ell}x_{\ell}\right)\)
    \\
    Distributed:&\vphantom{\xcancel{\pc{1}}}\xcancel{\pc{1}}
                                       & \(\left[x_{\ell}\right]_{1}\) & \(\longrightarrow\) & \(\left[x_{\ell}\right]_{1}+\xcancel{\left[\b{N}_{\ell}\right]_{1}\left[b_{\ell}-\b{A}_{\ell}x_{\ell}\right]_{1}} = \left[x_{\ell}\right]_{1}\) \\
            &\pc{2} & \(\left[x_{\ell}\right]_{2}\) & \(\longrightarrow\) & \(\left[x_{\ell}\right]_{2}+\left[\b{N}_{\ell}\right]_{2} \left[b_{\ell}-\b{A}_{\ell}x_{\ell}\right]_{2}\) \\
            &\pc{3} & \(\left[x_{\ell}\right]_{3}\) & \(\longrightarrow\) & \(\left[x_{\ell}\right]_{3}+\left[\b{N}_{\ell}\right]_{3}\left[b_{\ell}-\b{A}_{\ell}x_{\ell}\right]_{3}\)
      \end{tabular*}
  \caption{Schematic representation of a node failure during smoothing and the remedial action taken by the laissez-faire approach in the case of three compute nodes. \(\left[\bullet\right]_{i}\) represents the part of the quantity \(\bullet\) local to node \(i\).}
  \label{fig:faultIllustration}
\end{figure}

By the same token, the action of $\b{N}_\ell$ is prone to faults,
meaning that the true result $\b{N}_\ell\bc{X}_1\rho_\ell$ may be
corrupted and is therefore modelled by
$\bc{X}_2 \b{N}_\ell\bc{X}_1\rho_\ell$, where $\bc{X}_2$ is yet
another random diagonal matrix. The matrix $\b{N}_\ell$ corresponding
to damped Jacobi is diagonal and hence
$\bc{X}_2 \b{N}_\ell\bc{X}_1 = \bc{X}^{(S)} \b{N}_\ell$, where
$\bc{X}^{(S)}=\bc{X}_1\bc{X}_2$ is again a random diagonal matrix.
Consequently, the combined effect of the two sources of error can be
modelled by a single random diagonal matrix.

In summary, our model for the action of a smoothener prone to faults consists
of replacing the true smoothener $S$ used in the pre- and post-smoothing step
in the multigrid algorithm by the \emph{non-deterministic smoothener}
\begin{align}\label{eq:ndsmoothener}
	\c{S}_\ell(b_\ell, x_\ell) =
		x_\ell + \bc{X}^{(S)}_\ell \b{N}_\ell (b_\ell - \b{A}_\ell x_\ell)
\end{align}
in which $\bc{X}^{(S)}_\ell$ is a random diagonal matrix which models
the effect of the random faulty nature of the underlying hardware. The
model \cref{eq:ndsmoothener} tacitly assumes that the current iterate
\(x_{\ell}\) remains fault-free. We illustrate the remedial action to
a node failure in \Cref{fig:faultIllustration}.

\subsection{Restriction, Prolongation and Coarse Grid Correction}
\label{sec:restr-prol-coarse}

The restriction of the residual described by the step
\begin{align}\label{eq:dell}
	d_{\ell-1} = \b{R}(b_\ell - \b{A}_\ell x_\ell)
\end{align}
is prone to faults. Firstly, as in the case of the smoothener, the true
residual $\rho_\ell$ is prone to corruption and is modelled by
$\bc{X}^{(\rho)}_\ell\rho_\ell$. The resulting residual is then operated on by the
restriction $\b{R}$, which is itself prone to faults modelled using a random
diagonal matrix $\bc{X}^{(R)}_{\ell-1}$. We arrive at the following model for the
effect of faults on~\cref{eq:dell}:
\begin{align*}
	d_{\ell-1} = \bc{X}^{(R)}_{\ell-1} \b{R}\bc{X}^{(\rho)}_\ell
	(b_\ell - \b{A}_\ell x_\ell).
\end{align*}

The coarse grid correction $e_{\ell-1}$ is obtained by performing $\gamma$
iterations of $\c{\MG}_{\ell-1}$ with data $d_{\ell-1}$ and a zero initial
iterate. The effect of faults when applying the prolongation to $e_{\ell-1}$ is
also modelled by a random diagonal matrix $\bc{X}^{(P)}_\ell$ leading to the
following model for the effect of faults on the coarse grid correction and
prolongation steps:
\begin{align*}
	x_\ell \leftarrow x_\ell + \bc{X}^{(P)}_\ell \b{P}e_{\ell-1}.
\end{align*}

\subsection{Model for Multigrid Algorithm in Presence of Faults}
\label{sec:model-mult-algor}

Replacing each of the steps in the fault-free Multigrid
Algorithm~\ref{alg:multigrid} with their non-deterministic equivalent
yields the following model for the Fault-Prone Multigrid
Algorithm~\ref{alg:ndmg} and defines the associated fault-prone
multilevel operators $\c{\MG}_\ell$, $\ell\in\NN$.
\begin{algorithm}
  \begin{algorithmic}[1]
    \Require Right hand side $b_\ell$; Initial iterate $x_\ell$
    \Ensure $\c{\MG}_\ell(b_\ell, x_\ell)$
    \If{$\ell=0$}
    \Return $\b{A}_{0}^{-1}b_{0}$  \Comment{Exact solve on coarsest grid}
    \Else
    \For{$i\leftarrow 1$ \To{} $\nu_{1}$}
    \State $x_\ell\leftarrow \c{S}_\ell\left(b_\ell,x_\ell\right)$  \Comment{$\nu_1$ pre-smoothing steps}
    \EndFor
    \State $d_{\ell-1}\leftarrow\bc{X}^{(R)}_{\ell-1}\b{R}\bc{X}^{(\rho)}_\ell\left(b_\ell-\b{A}_\ell x_\ell\right)$  \Comment{Restriction to coarser grid}
    \State $e_{\ell-1}\leftarrow 0$
    \For{$j\leftarrow 1$ \To{} $\gamma$}
    \State $e_{\ell-1}\leftarrow \c{\MG}_{\ell-1}(d_{\ell-1},e_{\ell-1})$  \Comment{$\gamma$ coarse grid correction steps}
    \EndFor
    \State $x_\ell\leftarrow x_\ell + \bc{X}^{(P)}_\ell\b{P}e_{\ell-1}$  \Comment{Prolongation to finer grid}
    \For{$i\leftarrow 1$ \To{} $\nu_{2}$}
    \State $x_\ell\leftarrow \c{S}_\ell\left(b_\ell,x_\ell\right)$  \Comment{$\nu_2$ post-smoothing steps}
    \EndFor
    \EndIf
  \end{algorithmic}
  \caption{Model for Fault-Prone Multigrid Algorithm $\c{\MG}_\ell$ where $\bc{X}^{(\bullet)}_\ell$ are random diagonal matrices.}
  \label{alg:ndmg}
\end{algorithm}

The classical approach to the analysis of iterative solution methods for
linear systems uses the notion of an iteration matrix. For example, in the
case of the fault-free smoothening step~\cref{eq:smoothener}
\begin{align}\label{eq:defitmatrix}
	x - S_\ell(\b{A}_\ell x, y) = \b{E}_\ell^{S}(x-y),
        \quad\forall x,y\in\RR^{n_\ell}
\end{align}
where $\b{E}_\ell^{S}=\b{I}- \b{N}_\ell \b{A}_\ell$ is the iteration
matrix. The corresponding result for the fault-prone
smoothener~\cref{eq:ndsmoothener} is given by
\begin{align}\label{eq:defSitmatrix}
	x - \c{S}_\ell(\b{A}_\ell x, y)
	= \bc{E}_\ell^{S} (x-y),
        \quad\forall x,y\in\RR^{n_\ell}
\end{align}
where the iteration matrix $\bc{E}_\ell^{S}=\b{I}-\bc{X}_\ell^{(S)}
\b{N}_\ell \b{A}_\ell$ is now \emph{random}.  By analogy
with~\cref{eq:defitmatrix}, we define the iteration matrix for the Fault-Prone
Multigrid Algorithm~\ref{alg:ndmg} by the equation
\begin{align}\label{eq:defMGitmatrix}
  x - \c{\MG}_\ell(\b{A}_\ell x, y) = \bc{E}_\ell(x-y),
  \quad\forall x,y\in\RR^{n_\ell}.
\end{align}
In particular, $x - \c{\MG}_{0}(\b{A}_0 x,\bullet)=0$ and hence
$\bc{E}_{0}=\b{0}$, i.e. the zero matrix.

The $\nu_1$ pre-smoothing steps in \Cref{alg:ndmg} can be expressed in
the form
\begin{align*}
   x^{(0)} = x_\ell; \quad
   x^{(i)} = \c{S}_\ell(b_\ell, x^{(i-1)}), \quad i=1,\ldots,\nu_1.
\end{align*}
Using~\cref{eq:defSitmatrix} gives
\begin{align*}
   x-x^{(i)}_\ell = \bc{E}^{S}_\ell(x-x^{(i-1)}_\ell), \quad i=1,\ldots,\nu_1
\end{align*}
and hence
\begin{align}\label{eq:presmooth}
   x-x^{(\nu_1)}_\ell = \left(\bc{E}^{S}_\ell\right)^{\nu_1} (x-x_\ell).
\end{align}
Using this notation means that the vector $d_{\ell-1}$ appearing in
\Cref{alg:ndmg} is given by
\begin{align*}
   d_{\ell-1}
   = \bc{X}^{(R)}_{\ell-1} \b{R}\bc{X}^{(\rho)}_\ell
      (b_\ell-\b{A}_\ell x^{(\nu_1)}_\ell)
   = \bc{X}^{(R)}_{\ell-1} \b{R}\bc{X}^{(\rho)}_\ell \b{A}_\ell
      (x-x^{(\nu_1)}_\ell).
\end{align*}
where $b_\ell=\b{A}_\ell x$.

The coarse grid correction steps of \Cref{alg:ndmg} can be written in
the form
\begin{align*}
     e^{(0)}_{\ell-1} = 0; \quad
     e^{(j)}_{\ell-1} = \c{\MG}_{\ell-1}(d_{\ell-1},e^{(j-1)}_{\ell-1}),
	\quad j=1,\ldots,\gamma.
\end{align*}
Let $z=\b{A}_{\ell-1}^{-1}d_{\ell-1}$ then
\begin{align*}
    z - e^{(j)}_{\ell-1}
    = z - \c{\MG}_{\ell-1}(d_{\ell-1},e^{(j-1)}_{\ell-1})
    = \bc{E}_{\ell-1}(z - e^{(j-1)}_{\ell-1}), \quad j=1,\ldots,\gamma
\end{align*}
thanks to~\cref{eq:defMGitmatrix}. Iterating this result and recalling that
$e^{(0)}_{\ell-1}=0$, gives $z - e^{(\gamma)}_{\ell-1} = \bc{E}_{\ell-1}^\gamma
z$, or, equally well,
\begin{align*}
    e^{(\gamma)}_{\ell-1} =
       (\b{I}-\bc{E}_{\ell-1}^\gamma)\b{A}_{\ell-1}^{-1}d_{\ell-1}.
\end{align*}

Using these notations, the prolongation to the finer grid step in
\Cref{alg:ndmg} takes the form
\begin{align*}
    x^{(P)}_\ell = x^{(\nu_1)}_\ell + \bc{X}^{(P)}_\ell \b{P}e^{(\gamma)}_{\ell-1}
\end{align*}
and hence, by collecting results, we deduce that
\begin{align*}
    x-x^{(P)}_\ell
    &= (x-x^{(\nu_1)}_\ell) - \bc{X}^{(P)}_\ell \b{P}e^{(\gamma)}_{\ell-1} \\
    &= \left[ \b{I}
    - \bc{X}^{(P)}_\ell\b{P}( \b{I} -\bc{E}^\gamma_{\ell-1} )
	  \b{A}_{\ell-1}^{-1}\bc{X}^{(R)}_{\ell-1} \b{R}\bc{X}^{(\rho)}_\ell \b{A}_\ell
    \right](x-x^{(\nu_1)}_\ell).
\end{align*}
Arguments identical to those leading to~\cref{eq:presmooth} give the following
result
\begin{align*}
   x-\c{\MG}_{\ell}(\b{A}_\ell x,x_\ell)
      = \left(\bc{E}^{S}_\ell\right)^{\nu_2}(x-x^{(P)}_\ell).
\end{align*}
which, in view of identity~\cref{eq:defMGitmatrix} and~\cref{eq:presmooth},
yields the following recursive formula for the iteration matrix of
the Fault-Prone Multigrid Algorithm~\ref{alg:ndmg}:
\begin{align}\label{eq:Mrecursion}
    \bc{E}_\ell
    = \left(\bc{E}^{S,\text{post}}_\ell\right)^{\nu_2}
      \left[ \b{I}
    - \bc{X}^{(P)}_\ell\b{P}( \b{I} -\bc{E}^\gamma_{\ell-1} )
	  \b{A}_{\ell-1}^{-1}\bc{X}^{(R)}_{\ell-1} \b{R}\bc{X}^{(\rho)}_\ell \b{A}_\ell
    \right]
    \left(\bc{E}^{S,\text{pre}}_\ell\right)^{\nu_1},
\end{align}
for $\ell=1,\ldots,L$ with $\bc{E}_0=\b{0}$. Here, we have used
superscripts pre and post to reflect that the pre- and
post-smootheners are independent realisations of the same random
matrix.

By setting $\bc{E}_{L-1}=\b{0}$ and applying the
recursion~\cref{eq:Mrecursion} in the case $\ell=L$ yields a formula
for the iteration matrix of the \emph{Fault-Prone Two Grid Algorithm}:
\begin{align}\label{eq:twolevel}
    \bc{E}^{TG}_L
    &= \left(\bc{E}^{S,\text{post}}_L\right)^{\nu_2}
      \left[ \b{I} - \bc{X}^{(P)}_L\b{P}
	  \b{A}_{L-1}^{-1}\bc{X}^{(R)}_{L-1} \b{R}\bc{X}^{(\rho)}_{L} \b{A}_L
    \right]
    \left(\bc{E}^{S,\text{pre}}_L\right)^{\nu_1} \\
  &=\left(\bc{E}^{S,\text{post}}_L\right)^{\nu_2}
  \bc{E}_{L}^{CG}
  \left(\bc{E}^{S,\text{pre}}_L\right)^{\nu_1},\nonumber
\end{align}
corresponding to using an exact solver on level $L-1$.
Here \(\bc{E}^{CG}_{L}\) is the iteration matrix of the exact fault-prone coarse grid correction.

\subsection{Lyapunov Spectral Radius and Replica Trick}
\label{sec:lyap-spectr-radi}

If the Fault-Prone Multigrid, \Cref{alg:ndmg}, is applied using a
starting iterate with error $e^{(0)}$, then the error after $N\in\NN$
iterations is given by $e^{(N)}=\bc{E}_L^N e^{(0)}$. The matrix
$\bc{E}_L$ defined recursively by~\cref{eq:Mrecursion} is random, and
the product $\bc{E}_L^N$ should be interpreted as a product of $N$
independent samples of the matrix $\bc{E}_L$. The Fault-Prone
Multigrid iteration will converge provided that
\begin{align*}
  \lim_{N\rightarrow\infty} \norm{\bc{E}_L^N e^{(0)}}{}^{1/N} < 1
  \quad\text{a.s.}
\end{align*}
where $\norm{\bullet}{}$ is any norm on the (finite dimensional) space
$\RR^{n_L}$. The limit is understood in the sense of almost sure \emph{a.s.}
convergence for a given initial iterate $e^{(0)}$. A classic result, due to
Furstenberg and Kesten~\cite{FurstenbergKesten1960_ProductsRandomMatrices}, states that if
$\Exp \left[\log\left(\bc{E}_L\right)^{+}\right]<\infty$ then the following limit exists
\begin{align*}
  \lim_{N\rightarrow\infty} \norm{\bc{E}_L^N}{}^{1/N} = \varrho(\bc{E}_L)
  \quad\text{a.s.}
\end{align*}
where the non-random quantity
\begin{align*}
 \varrho(\bc{E}_L)=\lim_{N\rightarrow\infty}\exp\left\{\Exp\left[\log\norm{\bc{E}_{L}^{N}}{}^{1/N}\right]\right\}
\end{align*}
is the \emph{Lyapunov spectral radius}. Further
details and properties of the Lyapunov spectral radius will be found in the
books by Bougerol and Lacroix~\cite{BougerolLacroix1985_ProductsRandomMatricesWith} or Crisanti et
al.~\cite{CrisantiPaladinEtAl1993_ProductsRandomMatrices}.

Suppose the matrix $\bc{E}_L$ is non-random, then the Lyapunov
spectral radius coincides with the usual spectral radius
$\rho\left(\bc{E}_L\right)$, and it extends the notion of an
asymptotic rate of convergence to random iterative schemes. In
particular, $\varrho\left(\bc{E}_L\right)<1$ means that the matrix
product is almost surely convergent, whereas
$\varrho\left(\bc{E}_L\right)>1$ means that it is almost surely
divergent.

Consequently, the convergence analysis of the Fault-Prone Multigrid
Method, \Cref{alg:ndmg}, boils down to the study of the Lyapunov
spectral radius of the random iteration matrix $\bc{E}_L$.
Unfortunately, the treatment of the Lyapunov spectral radius is
considerably more complex than dealing with the usual spectral radius.
In general its determination is a NP-hard
problem~\cite{TsitsiklisBlondel1997_LyapunovExponentJointSpectral}.
Since the Lyapunov spectral radius is not necessarily a continuous
function of the fault
probabilities~\cite{AinsworthGlusa2015_NumericalMathematicsAdvancedApplicationsEnumath2015},
straightforward perturbation analysis cannot be employed. The
following result is sometimes referred to as the \emph{Replica Trick}:
\begin{restatable}[Replica trick \cite{CrisantiPaladinEtAl1993_ProductsRandomMatrices}]{theorem}{replicaTrickLemma}
  \label{lem:replicaTrick}
  Let \(\bc{E}\) be a random square matrix.
  Then
  \begin{align*}
    \varrho\left(\bc{E}\right) & \leq \sqrt{\rho\left(\Exp\left[ \bc{E}\otimes\bc{E}\right]\right)},
  \end{align*}
  where $\rho(\bullet)$ denotes the usual spectral radius.
\end{restatable}
We refer the reader to \Cref{sec:preliminaries} for the proof of \Cref{lem:replicaTrick}.
One attractive feature of this estimate is the appearance of the usual spectral
radius of the non-random matrix $\Exp \left[\bc{E}_L\otimes\bc{E}_L\right]$ rather than the
Lyapunov spectral radius which, however, comes at the price of having to
deal with the Kronecker product $\bc{E}_L\otimes\bc{E}_L$ which, as we shall
see later, presents its own difficulties.

\subsection{Behaviour of the Smoothener under the Model}
\label{sec:behav-smooth-under}

The iteration matrix for the fault-prone smoothener~\cref{eq:ndsmoothener} is given
by
\begin{align*}
   \bc{E}_\ell^{S}=\b{I}-\bc{X}_\ell^{(S)}\b{N}_\ell\b{A}_\ell,
\end{align*}
Here, we take $\bc{X}_\ell^{(S)}=\diag(\chi_1,\ldots,\chi_{n_\ell})$ to be a random diagonal
matrix of componentwise faults with $\Exp(\chi_j) = 1-q$ and $\Cov(\chi_j,\chi_{j})=q(1-q)\delta_{ij}$. The behaviour of
the fault-prone smoothener is governed by the Lyapunov spectral radius
$\varrho(\bc{E}_\ell^{S})$.

\begin{theorem}\label{thm:smoothener}
Let $\b{E}_\ell^{S}=\b{I}-\b{N}_\ell\b{A}_\ell$ be the iteration matrix for
any convergent smoothener.  Then, the corresponding fault-prone smoothener
satisfies
\begin{align*}
   \varrho(\bc{E}_\ell^{S})\le\sqrt{
	   (1-q)\norm{\b{E}_\ell^{S}}{2}^2
	   + q\left(\norm{\b{E}_\ell^{S}}{2}+\norm{\b{N}_\ell\b{A}_\ell}{2}\right)^2
   },
\end{align*}
for all $q\in[0,1]$.
\end{theorem}
\begin{proof}
  Let $\bc{B}=\bc{X}^{(S)}\b{N}\b{A}$ and dispense with subscripts and superscripts for the
  duration of the proof.
  Then
  \begin{align}
    \Exp\left[\bc{E}\otimes\bc{E}\right] &= \Exp\left[\b{I}\otimes\b{I}-\bc{B}\otimes\b{I}-\b{I}\otimes\bc{B}+\bc{B}\otimes\bc{B}\right] \nonumber\\
    &=\b{I}\otimes\b{I} -\Exp\left[\bc{B}\right]\otimes\b{I}-\b{I}\otimes\Exp\left[\bc{B}\right]+\Exp\left[\bc{B}\right]\otimes\Exp\left[\bc{B}\right] \nonumber\\
    &\quad- \Exp\left[\bc{B}\right]\otimes\Exp\left[\bc{B}\right] + \Exp\left[\bc{B}\otimes \bc{B}\right] \nonumber\\
    &=\Exp\left[\bc{E}\right]\otimes\Exp\left[\bc{E}\right] + \Var\left[\bc{B}\right]\label{eq:oneRandomMatrix}
  \end{align}
  and
  \begin{align*}
    \Var\left[\bc{B}\right]
    = \Var\left[\bc{X}^{(S)}\right]\left(\b{NA}\right)\otimes\left(\b{NA}\right)
    = q(1-q) \b{K}\left(\b{NA}\right)\otimes\left(\b{NA}\right),
  \end{align*}
  where
  \begin{align*}
    \b{K} = \blockdiag\left(\vec{e}^{\,(j)}\otimes\left(\vec{e}^{\,(j)}\right)^T,~j=1,\ldots,n\right)
  \end{align*}
  and $\vec{e}^{\,(j)}$ the $j$-th canonical unit vector in \(\mathbb{R}^{n}\). Hence, since
  $\norm{\b{C}\otimes\b{C}}{2}=\norm{\b{C}}{2}^2$ and
  $\norm{\b{K}_\ell\b{C}}{2}\leq\norm{\b{C}}{2}$ for any compatible matrix
  $\b{C}$, we obtain
  \begin{align*}
    \norm{\Var\left[\bc{B}\right]}{2}
    &= q(1-q)\norm{\b{K}\left(\b{NA}\right)\otimes\left(\b{NA}\right)}{2}
    \le q(1-q)\norm{\b{NA}}{2}^2
  \end{align*}
  and
  \begin{align*}
    \norm{\Exp\left[\bc{E}\right]\otimes\Exp\left[\bc{E}\right]}{2}
    = \norm{\Exp\left[\bc{E}\right]}{2}^2
    = \norm{\b{I}-(1-q)\b{NA}}{2}^2.
  \end{align*}
  Consequently,
  \begin{align*}
    \norm{\Exp\left[\bc{E}\otimes\bc{E}\right]}{2}
    &\le \norm{\b{I}-(1-q)\b{NA}}{2}^2 + q(1-q)\norm{\b{NA}}{2}^2 \\
    &= \norm{\b{E}+q\b{NA}}{2}^2 + q(1-q)\norm{\b{NA}}{2}^2 \\
    &\le \left(\norm{\b{E}}{2}+q\norm{\b{NA}}{2}\right)^2 + q(1-q)\norm{\b{NA}}{2}^2 \\
    &= (1-q)\norm{\b{E}}{2}^2+q(\norm{\b{E}}{2}+\norm{\b{NA}}{2})^2
  \end{align*}
  and the result follows thanks to \Cref{lem:replicaTrick}.
\end{proof}

In order to illustrate \Cref{thm:smoothener}, we consider the situation
where the smoothener is taken to be a convergent relaxation scheme so that
$\norm{\b{E}^{S}_{\ell}}{2}=\gamma\in[0,1)$. The triangle inequality gives
$\norm{\b{N}_\ell \b{A}_\ell}{2}\leq 1+\gamma$ and hence, thanks to
\Cref{thm:smoothener}, we deduce that
\begin{align*}
    \varrho(\bc{E}_\ell^{S})\le\sqrt{\gamma^2 + q(1+4\gamma+3\gamma^{2})}.
\end{align*}
As a consequence, we find that a sufficient condition for the fault-prone smoothener
to remain convergent is that the error rate is sufficiently small: $q<(1-\gamma)/(1+3\gamma)$.

\section{Summary of Main Results}
\label{sec:summary-main-results}

Let the following standard assumptions for the convergence of the fault-free multigrid method be satisfied~\cite{Braess2007_FiniteElements,Hackbusch1985_MultiGridMethodsApplications,Hackbusch1994_IterativeSolutionLargeSparseSystemsEquations,TrottenbergOosterleeEtAl2001_Multigrid}:
\begin{assumption}[Smoothing property]\label{as:smoothing}
  There exists \(\eta: \mathbb{N}\rightarrow \mathbb{R}_{\geq0}\) satisfying \(\lim_{\nu\rightarrow\infty}\eta(\nu)  = 0\) and such that
  \begin{align*}
    \norm{\b{A}_{\ell}\left(\b{E}^{S}_{\ell}\right)^{\nu}}{2} & \leq \eta(\nu)\norm{\b{A}_{\ell}}{2}, \quad \nu\geq0,~\ell=1,\dots,L.
  \end{align*}
\end{assumption}

\begin{assumption}[Approximation property]\label{as:approximation}
  There exists a constant \(C_{A}\) such that
  \begin{align*}
    \norm{\b{E}^{CG}_{\ell}\b{A}_{\ell}^{-1}}{2} & \leq \frac{C_{A}}{\norm{\b{A}_{\ell}}{2}}, \quad \ell=1,\dots,L.
  \end{align*}
\end{assumption}

\begin{assumption}\label{as:smoother}
  The smoothener is non-expansive, i.e.\ \(\rho\left(\b{E}^{S}_{\ell}\right)=\norm{\b{E}^{S}_{\ell}}{A}\leq 1\), and there exists a non-increasing function \(C_{S}:\mathbb{N}\rightarrow \mathbb{R}_{\geq0}\) such that
  \begin{align*}
    \norm{\left(\b{E}^{S}_{\ell}\right)^{\nu}}{2} & \leq C_{S}\left(\nu\right), \quad \nu\geq 1,~\ell=1,\dots,L.
  \end{align*}
\end{assumption}

\begin{assumption}\label{as:prolongation}
  There exist positive constants \(\underline{C}_{p}\) and \(\overline{C}_{p}\) such that
  \begin{align*}
    \underline{C}_{p}^{-1}\norm{x_{\ell}}{2} & \leq \norm{\b{P} x_{\ell}}{2}\leq \overline{C}_{p}\norm{x_{\ell}}{2} \quad \forall x_{\ell}\in \mathbb{R}^{n_{\ell}},~\ell=0,\dots,L-1.
  \end{align*}
\end{assumption}
\Cref{as:smoothing,as:approximation} guarantee convergence of the Two Grid Method, i.e. \(\norm{\b{E}^{TG}_{\ell}}{2}\leq C<1\), whereas \cref{as:smoother,as:prolongation} can be used to show multigrid convergence with respect to \(\norm{\bullet}{2}\).

\subsection{Analysis of the Fault-Prone Two Grid Method}
\label{sec:analysis-fault-prone}

We first analyse the Fault-Prone Two Grid Method in the following setting:
Let $\Omega\subset\mathbb{R}^{d}$ be a convex polyhedral domain,
$\Gamma_D\subseteq\partial\Omega$ be non-empty and
$V=\{v\in H^{1}(\Omega):v=0\mbox{ on }\Gamma_D\}$. Let $\c{T}_{0}$ be a
partitioning of $\Omega$ into the union of elements each of which is a
$d+1$-simplex, such that he intersection of any pair of distinct
elements is a single common vertex, edge, face or sub-simplex of both
elements. For $\ell\in\NN$, the partition $\c{T}_\ell$ is obtained by
uniformly subdividing each element in $\c{T}_{\ell-1}$ into
subsimplices such that $\c{T}_\ell$ satisfies the same conditions as
$\c{T}_0$ \cite{KroegerPreusser2008_Stability8TetrahedraShortest}. Let
$\PP_k$, $k\in\NN$, consist of polynomials on $\RR^d$ of total degree
at most $k$. The subspace $V_\ell$ consists of continuous piecewise
polynomials belonging to $\PP_k$ on each element on the partition
$\c{T}_\ell$. In particular, this means that the spaces are nested:
$V_0\subset V_1\subset\ldots\subset V_L \subset V$.

Let $a:V\times V\to\RR$ be a continuous, $V$-elliptic, bilinear form and
$F:V\to\RR$ be a continuous linear form and consider the problem of obtaining
$u_L\in V_L$ such that
\begin{align}\label{eq:fem}
  a(u_{L},v) = F(v)\quad\forall v\in V_L.
\end{align}
Let $\phi_\ell^{(i)}$, $i=1,\dots,n_\ell$, denote the usual nodal basis for
$V_\ell$, and $\phi_\ell$ be the corresponding vector formed using the basis
functions. Problem \cref{eq:fem} is then equivalent to the matrix equation
\begin{align*}
  \b{A}_L x_L = F_L
\end{align*}
where $\b{A}_\ell=a(\phi_\ell,\phi_\ell)$ is the stiffness matrix,
\(x_{L}\) is the coefficient vector of \(u_{\ell}\) and
$F_L = F(\phi_L)$ is the load vector. For future reference, we
define the mass matrix by $\b{M}_\ell=(\phi_\ell,\phi_\ell)$. The
nesting of the spaces $V_\ell\subset V_{\ell+1}$ means that there
exists a matrix $\b{R}^\ell_{\ell+1}$ such that
$\phi_\ell = \b{R}^\ell_{\ell+1}\phi_{\ell+1}$, i.e. the restriction
matrix with the associated prolongation matrix $\b{P}^{\ell+1}_\ell$
taken to be the transpose of $\b{R}^\ell_{\ell+1}$.

The first main result of the present work concerns the convergence of
the Fault-Prone Two-Grid Algorithm for the special case of all
\(\bc{X}_{l}^{(\bullet)}\) having independent Ber\-noulli distributed
diagonal entries with probability of fault \(q\) (which satisfies \Cref{as:faults} with \(\varepsilon=q\)):

\begin{restatable}{theorem}{twogrid}\label{thm:perturbedTGconv}
  Let \(\bc{E}^{TG}_{L} = \left(\bc{E}^{S,\text{post}}_{L}\right)^{\nu_{2}} \bc{E}^{CG}_{L} \left(\bc{E}^{S,\text{pre}}_{L}\right)^{\nu_{1}}\)
  be the iteration matrix of the Fault-Prone Two Grid Method with Jacobi smoothener and component-wise faults of rate \(q\) in prolongation, restriction, residual and in the smoothener, and let \(\b{E}^{TG}_{L}\) be its fault-free equivalent.
  Assume that the \Cref{as:smoothing,as:approximation,as:smoother,as:prolongation} hold.
  Then
  \begin{align}
    \varrho \left(\bc{E}^{TG}_{L}\right) & \leq \norm{\b{E}^{TG}_{L}}{A} +
                                       C\begin{cases}
                                         qn_{L}^{(4-d)/2d} & d < 4, \\
                                         q\sqrt{\log n_{L}} & d=4, \\
                                         q & d > 4,
                                       \end{cases} \label{eq:perturbedTGconv}
  \end{align}
  where \(\norm{\bullet}{A}\) is the matrix energy norm and the constant \(C\) is independent of \(q\) and \(L\).
\end{restatable}

The proof of \Cref{thm:perturbedTGconv} is postponed to \Cref{sec:proof-TGconv}.
If the probability of a fault occurring is zero, i.e. $q=0$, then the
bound~\cref{eq:perturbedTGconv} reduces to the quantity $\norm{\b{E}^{TG}_{L}}{A}$
which governs the convergence of the fault-free Two Grid Algorithm. Under the
foregoing assumptions on the finite element discretization, it is known
\cite{Bramble1993_MultigridMethods,
Hackbusch1985_MultiGridMethodsApplications,
TrottenbergOosterleeEtAl2001_Multigrid,
Braess2007_FiniteElements}
that $\norm{\b{E}^{TG}_{L}}{A}\le C_{TG}<1$, where $C_{TG}$ is a
positive constant independent of $L$.  If $q\in(0,1)$, then the
bound~\cref{eq:perturbedTGconv} depends on the failure probability $q$, the number of
unknowns $n_{L}$ and the spatial dimension $d$ of the underlying discretization.

\begin{wrapfigure}{i}{0pt}
  \centering
  \includegraphics{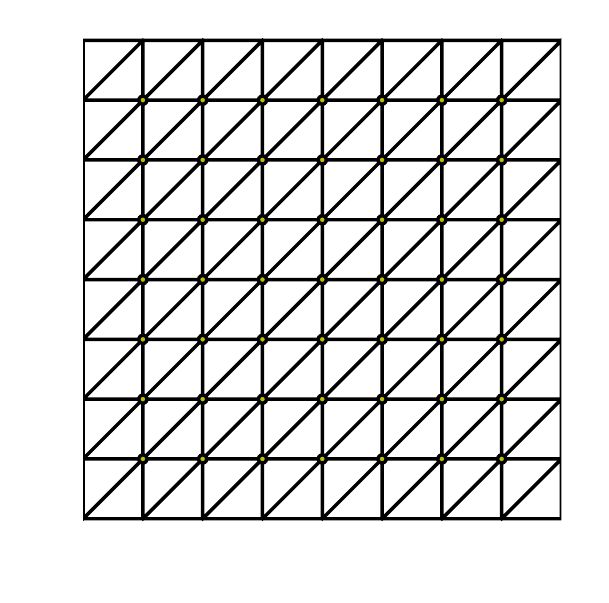}
  \caption{Uniform mesh for the square domain.}\label{fig:poisson2d-mesh2}
\end{wrapfigure}

In the simplest case, in which the spatial dimension $d>4$, the Fault-Prone Two
Grid Method will converge uniformly in $n_{L}$ provided that the failure
probability $q$ is sufficiently small, i.e. $q\in[0,(1-C_{TG})/C)\subset [0,1)$.
The estimate in the cases of most practical interest, in which $d<4$, is less
encouraging with the bound depending on the size $n_{L}$ of the problem being
solved. In particular, the estimate means that no matter how small the probability $q$
of a failure may be, as the size of the problem being solved
grows, the bound will exceed unity suggesting that the Two Grid Method will
fail to converge \emph{at all}.

\begin{figure}
  \centering
  \includegraphics{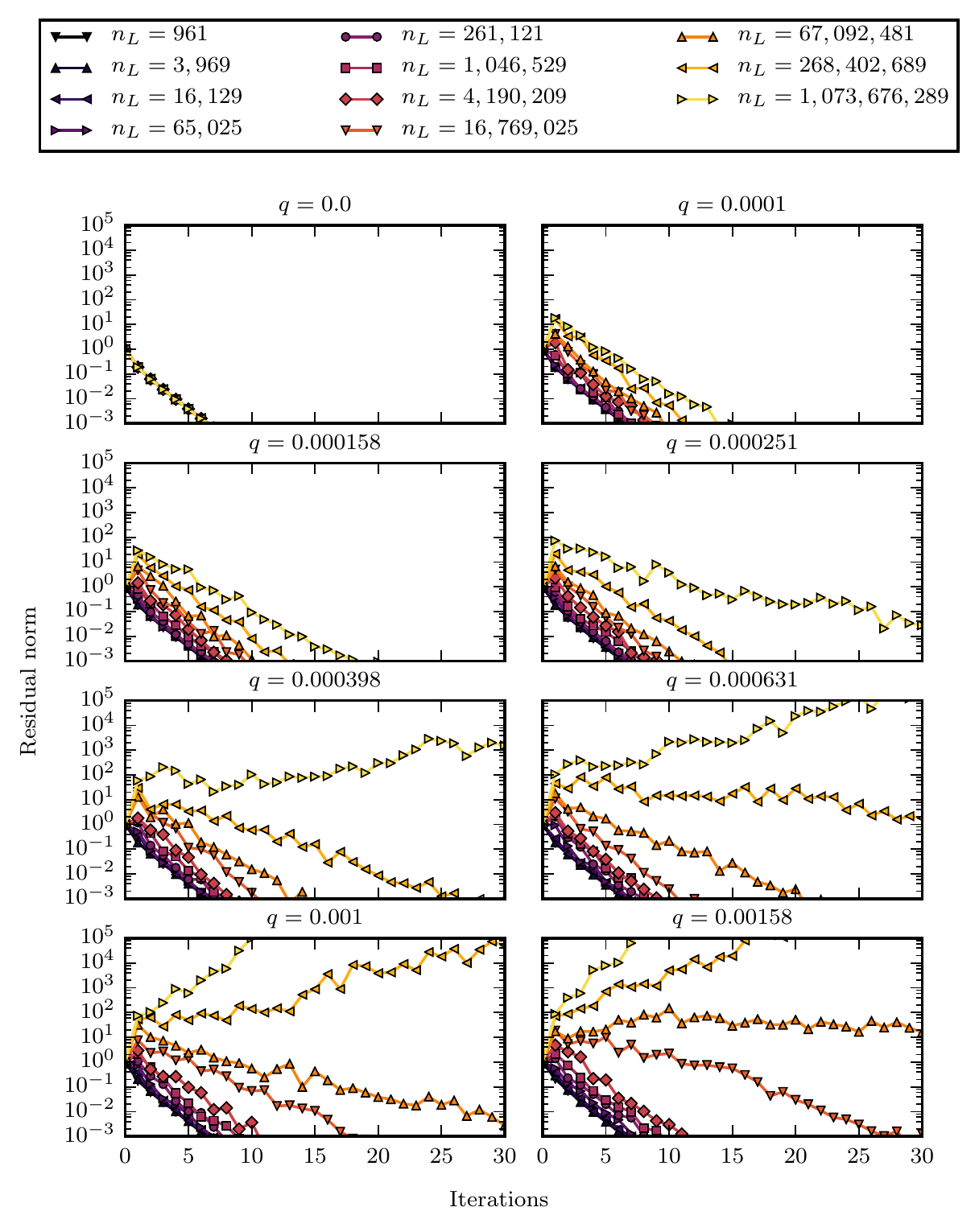}
  \caption{Plot of the norm of the residual against iteration number for the
    Fault-Prone Two Grid Method in the case of discretization of the Poisson problem
    on a square domain.} \label{fig:poisson2dgeoTG-residuals}
\end{figure}
In order to illustrate this behaviour and to test the sharpness of the
estimate~\cref{eq:perturbedTGconv}, we apply the Fault-Prone Two Grid
Method to the system arising from a discretization of the Poisson problem on a square
using piecewise linears on a uniform mesh of triangles (\Cref{fig:poisson2d-mesh2}).
In \Cref{fig:poisson2dgeoTG-residuals} we plot the norm of the residual
after each iteration for a range of
values of failure probability $q$ and system sizes $n_L$ ranging from 1
thousand to 1 billion. These computations were performed on the Titan
supercomputer located at Oak Ridge National Laboratory.


It is observed that when $q=0$, the method reaches residual norm
\(10^{-3}\) in about 7 iterations uniformly for all values of \(n_{L}\) up to
around one billion. However, as $q$ grows, it is observed that
the convergence deteriorates in all cases, with the deterioration
being more and more severe as the number of unknowns is increased, until,
eventually, the method fails to converge at all.

The fact that the Two Grid Method fails to converge even for such a
simple classical problem means one cannot expect a more favourable
behaviour for realistic practical problems.

In order to verify the scaling law suggested
by~\cref{eq:perturbedTGconv} we examine the case \(d=2\) further. In
particular, when $d=2$, we expect convergence to fail when
$q\sqrt{n_{L}}$ exceeds some threshold.
\Cref{fig:poisson2dgeoTG-qd-rhoL} shows a contour plot of the
variation of the Lyapunov spectral radius (obtained by 1000 iterations
of the Fault-Prone Two Grid Method) with respect to the failure probability $q$
and the size $n_{L}$ of the system. The plot indicates the boundary of
the region in which the Lyapunov spectral radius exceeds unity, or
equally, the region in which the fault-prone iteration no longer
converges at all. It will be observed that the lines on which
$q\sqrt{n_{L}}$ remains constant (also indicated in the plot) coincide
with the contours of the plot in the region in which failure to
converge occurs. This supports the scaling law suggested in the
estimate~\cref{eq:perturbedTGconv}.

\begin{figure}
  \centering
  \includegraphics{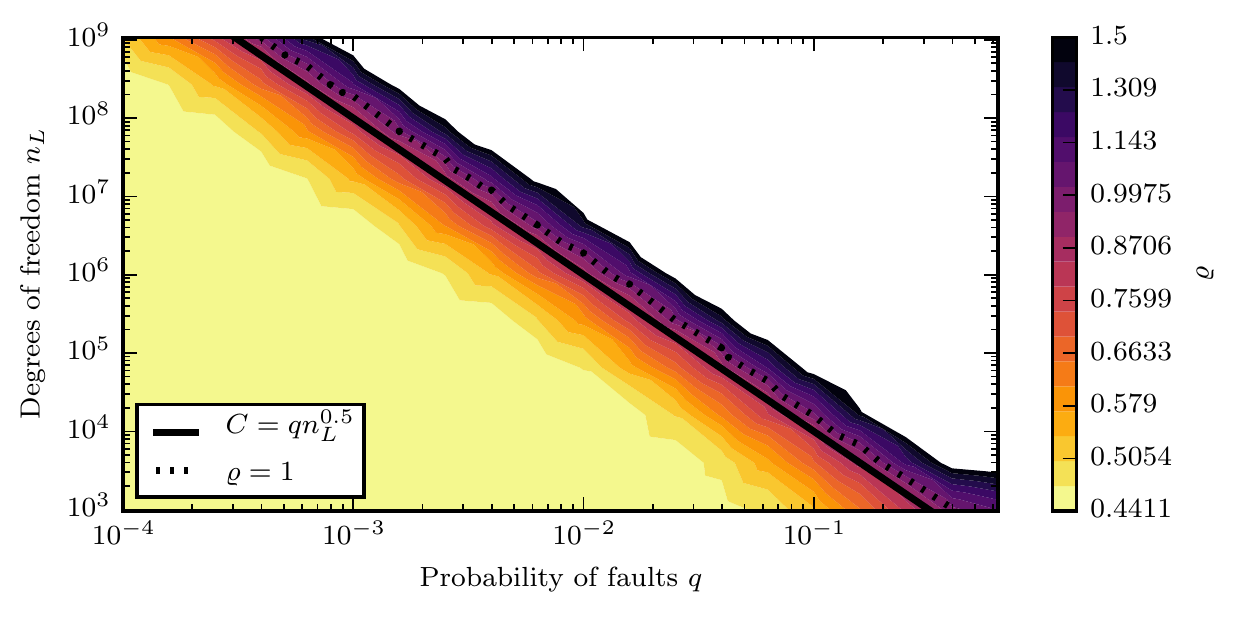}
  \caption{Lyapunov spectral radius $\varrho(\bc{E}^{TG}_L)$ for the iteration
    matrix for the Fault-Prone Two Grid Method in the case of discretization
    of the Poisson problem on a square domain.
  }\label{fig:poisson2dgeoTG-qd-rhoL}
\end{figure}

\subsection{Analysis of the Fault-Prone Two Grid Method with Minimal Protection}
\label{sec:analysis-fault-prone-1}

\Cref{thm:perturbedTGconv} indicates that the Fault-Prone Two Grid Method will generally not be resilient to faults.
What remedial action, in addition to laissez-faire, is needed to restore the uniform convergence of the Fault-Prone Multigrid Method to that of the fault-free scheme?
\Cref{thm:TGNoProlong} states that if the \emph{prolongation} operation is protected, meaning that \(\bc{X}_{L}^{(P)}=\b{I}\), then the rate of convergence of the Two Grid Method is independent of the number of unknowns and close to the rate of the fault-free method.

\begin{restatable}{theorem}{TGNoProlong}\label{thm:TGNoProlong}
  Let \(\bc{E}^{TG}_{L}\left(\nu_{1},\nu_{2}\right) = \left(\bc{E}^{S,\text{post}}_{L}\right)^{\nu_{2}} \bc{E}^{CG}_{L} \left(\bc{E}^{S,\text{pre}}_{L}\right)^{\nu_{1}}\)
  be the iteration matrix of the Fault-Prone Two Grid Method with faults in smoothener, residual and restriction, and protected prolongation.
  Provided \Cref{as:faults,as:smoothing,as:approximation,as:smoother,as:prolongation} hold, and that \(\b{N}_{L}\) and \(\bc{X}^{(S)}_{L}\) commute, we find that
  \begin{align*}
    \varrho \left(\bc{E}^{TG}_{L}\left(\nu_{1},\nu_{2}\right)\right) & \leq \min_{\substack{\mu_{1}+\mu_{2}=\nu_{1}+\nu_{2} \\\mu_{1},\mu_{2}\geq 0}}\norm{\b{E}^{TG}_{L}\left(\mu_{1},\mu_{2}\right)}{2} + C\varepsilon,
  \end{align*}
  where the constant \(C\) is independent of \(\varepsilon\) and \(L\).
\end{restatable}

The proof of \Cref{thm:TGNoProlong} will be found in \Cref{sec:proof-TGNoProlong}.
\Cref{thm:TGNoProlong} makes no additional assumptions about the origin of the problem or the solver hierarchy, and allows for more general fault patterns and smootheners than \Cref{thm:perturbedTGconv} and as such, is more generally applicable.

We apply the Fault-Prone Two Grid Method incorporating protection of the prolongation to the same test problem as above.
In \Cref{fig:poisson2dgeoTGNoProlong-residuals} we plot the evolution of the residual norm for varying system sizes and fault rates, and in \Cref{fig:poisson2dgeoTGNoProlong-qd-rhoL} we plot the Lyapunov spectral radius of the iteration matrix.
Observe that the rate of convergence is independent of the number of unknowns and even stays below unity (indicating convergence) for fault probabilities as high as \(q\approx 0.6\).
Moreover, \Cref{fig:poisson2dgeoTGNoProlong-residuals} seems to indicate that while the asymptotic behaviour is independent of the number of unknowns, the first iteration is not.

In practice, protection of the prolongation might be achieved using standard techniques, such as replication or checkpointing.
Since the prolongation is a local operation, its protected application could be overlapped with post-smoothing to minimize the overall performance penalty.
Alternatively, some developers have proposed architectures on which certain components are less susceptible to faults.
In such a case, the prolongation operator would be a candidate for this treatment.

\begin{figure}
  \centering
  \includegraphics{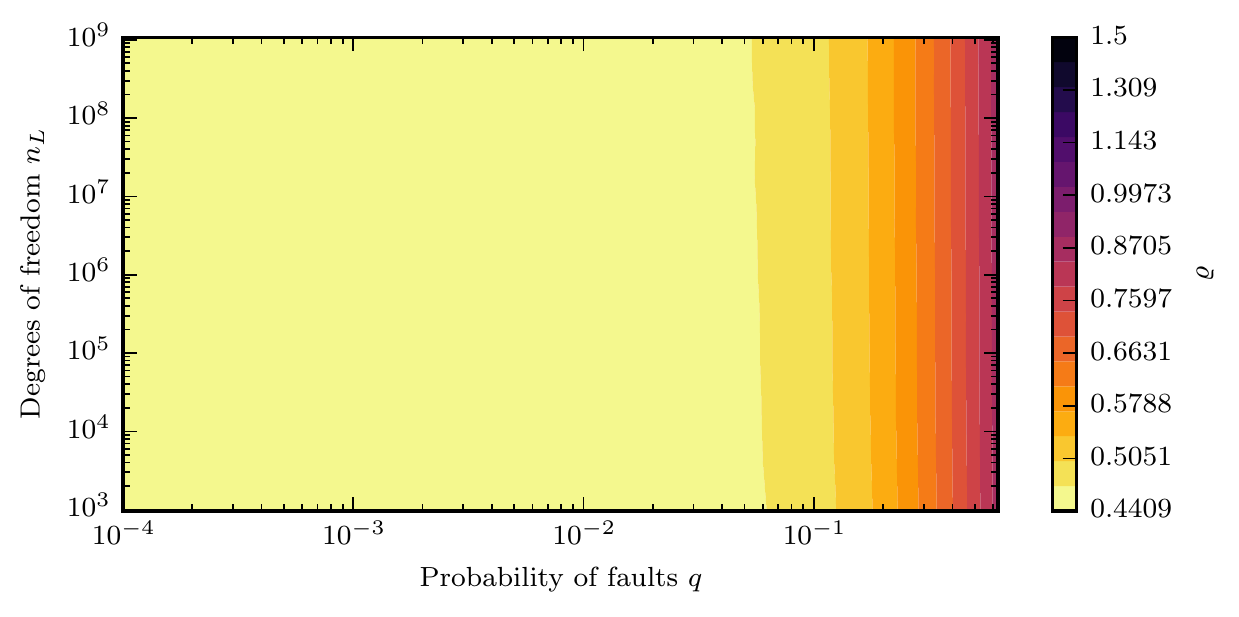}
  \caption{Lyapunov spectral radius $\varrho(\bc{E}^{TG}_L)$ for the
    iteration matrix for the Fault-Prone Two Grid Method with protected
    prolongation in the case of discretization of Poisson problem on a
    square domain.
  }\label{fig:poisson2dgeoTGNoProlong-qd-rhoL}
\end{figure}

\begin{figure}
  \centering
  \includegraphics{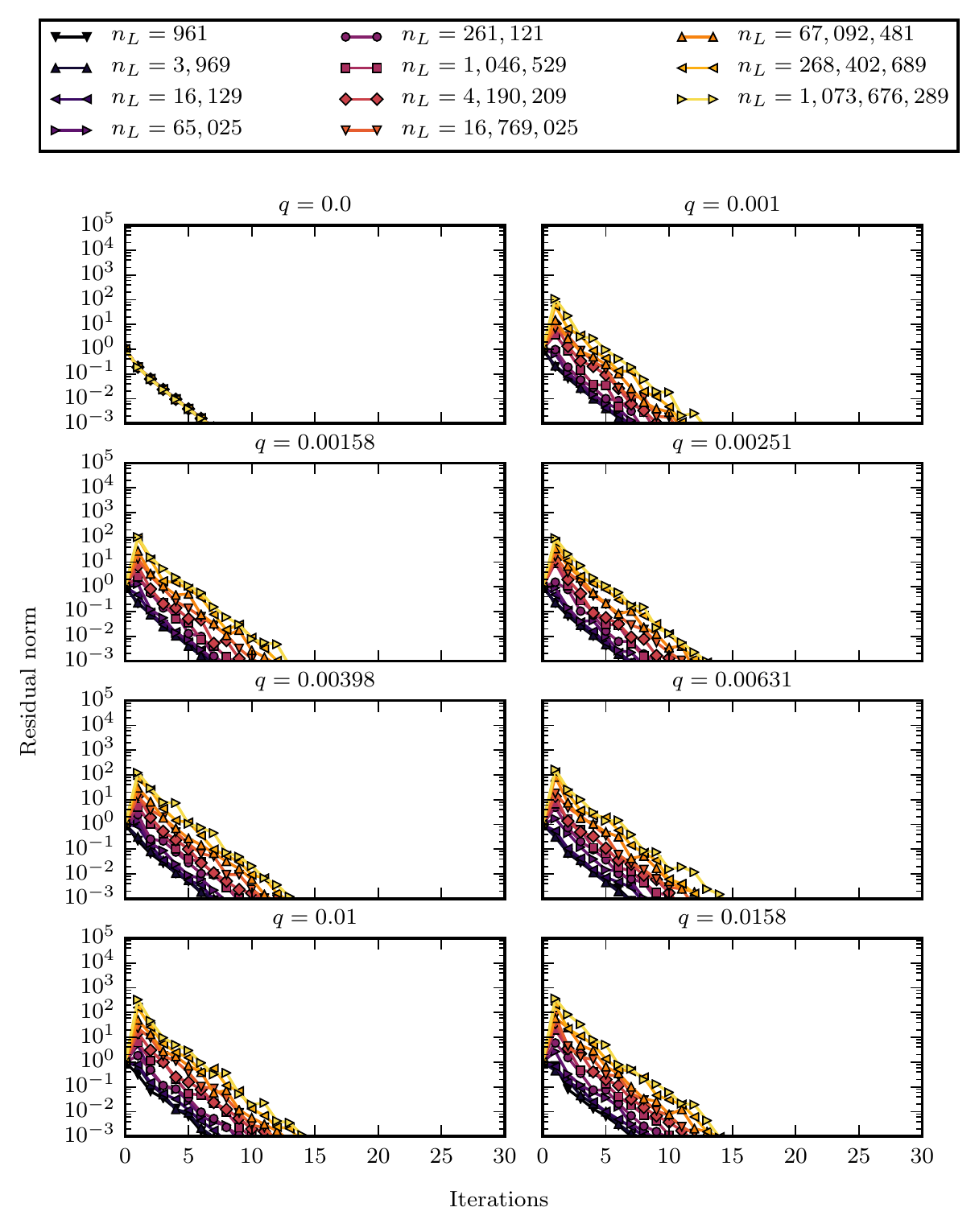}
  \caption{Plot of the norm of the residual against iteration number
    for the Fault-Prone Two Grid Method with protected prolongation in
    the case of discretization of Poisson problem on square domain.} \label{fig:poisson2dgeoTGNoProlong-residuals}
\end{figure}

In the present work we have proposed a simple approach to fault
mitigation for iterative methods. Analysis and numerical examples
showed that the Two Grid Method is not fault resilient no matter how
small the error rate \(q\) may be. Moreover, we addressed the question
of what is the minimal action required to restore the convergence of
the algorithm. We showed that this shortcoming can be overcome by
protection of the prolongation operation using standard techniques.
Forthcoming work will address these issues in the case of the
multigrid method.

\appendix

\section{Preliminaries}
\label{sec:preliminaries}

Throughout the appendices, \(C\) will be a generic constant whose value can change from line to line, but which is independent of \(\ell\), \(n_{\ell}\), \(q\) and \(\varepsilon\).

\begin{definition}
  Let \(\b{Z},~\b{Y}\in \mathbb{R}^{n \times m}\), \(n,~m\in\mathbb{N}\).
  The \emph{Kronecker product} \(\b{Z}\otimes \b{Y}\in\mathbb{R}^{n^{2}\times m^{2}}\) and the \emph{elementwise} or \emph{Hadamard product} \(\b{Z}\circ \b{Y}\in\mathbb{R}^{n\times m}\) of \(\b{Z}\) and \(\b{Y}\) are defined as
  \begin{align*}
    \left(\b{Z}\otimes \b{Y}\right)_{in+p,jm+q} &= \b{Z}_{ij} \b{Y}_{pq},
    & \left(\b{Z}\circ \b{Y}\right)_{ij} &= \b{Z}_{ij} \b{Y}_{ij}
  \end{align*}
  for \(1\leq i,p\leq n\), \(1\leq j,q\leq m\), and the \(k\)-th \emph{Kronecker power} \(\b{Z}^{\otimes k}\in\mathbb{R}^{n^{k}\times m^{k}}\) and the \(k\)-th \emph{Hadamard power} \(\b{Z}^{\circ k}\in\mathbb{R}^{n\times m}\) of \(\b{Z}\) as
  \begin{align*}
    \b{Z}^{\otimes k} &:= \underbrace{\b{Z}\otimes \b{Z} \otimes \cdots \otimes \b{Z}}_{k \text{ times}},
    & \b{Z}^{\circ k} &:= \underbrace{\b{Z}\circ \b{Z} \circ \cdots \circ \b{Z}}_{k \text{ times}}.
  \end{align*}
\end{definition}

\replicaTrickLemma*
\begin{proof}
  Let \(\b{E}\) be a non-random square matrix, then
  \begin{align*}
    \norm{\b{E}}{F}^{2}
    & = \operatorname{tr}\left(\b{EE}^{T}\right)
     = \operatorname{vec}\left(\b{I}\right) \cdot \operatorname{vec}\left(\b{EE}^{T}\right) \\
    & = \operatorname{vec}\left(\b{I}\right) \cdot \operatorname{vec}\left(\b{E}\b{I}\b{E}^{T}\right)
     = \operatorname{vec}\left(\b{I}\right) \cdot \b{E}^{\otimes2} \operatorname{vec}\left(\b{I}\right),
  \end{align*}
  where \(\norm{\bullet}{F}\) denotes the Frobenius norm and \(\operatorname{vec}\left(\bullet\right)\) denotes the vectorization of a matrix by stacking its columns.
  Hence, if \(\lambda_{j,\otimes 2}\) and \(\vec{v}_{j,\otimes 2}\) denote the eigenvalues and eigenvectors of \(\Exp \left[\bc{E}^{\otimes2}\right]\), sorted in descending order with respect to their absolute value, we have
  \begin{align*}
    \Exp \left[\norm{\bc{E}^{N}}{F}^{2}\right]
    & = \operatorname{vec}\left(\b{I}\right) \cdot \Exp\left[\left(\bc{E}^{\otimes2}\right)^{N}\right] \operatorname{vec}\left(\b{I}\right) \\
    & = \operatorname{vec}\left(\b{I}\right) \cdot \Exp\left[\bc{E}^{\otimes2}\right]^{N} \operatorname{vec}\left(\b{I}\right)
     = \left[\operatorname{vec}\left(\b{I}\right) \cdot \vec{v}_{1,\otimes 2}\right]^{2}\lambda_{1,\otimes 2}^{N} + \c{O}\left(\lambda_{2,\otimes 2}^{N}\right).
  \end{align*}
  Assume without loss of generality that \(\left|\lambda_{1,\otimes2}\right|>\left|\lambda_{2,\otimes2}\right|\).
  By Jensen's inequality, we find
  \begin{align*}
    \varrho\left(\bc{E}\right)&=\lim_{N\rightarrow\infty}\exp \left(\frac{1}{N}\Exp\left[\log\norm{\bc{E}^{N}}{F}\right]\right)
                              \leq \lim_{N\rightarrow\infty}\exp \left(\frac{1}{2N}\log\Exp\left[\norm{\bc{E}^{N}}{F}^{2}\right]\right) \\
                              &=\sqrt{\lambda_{1,\otimes 2}}
                                = \sqrt{\rho\left(\Exp\left[ \bc{E}^{\otimes 2}\right]\right)}
  \end{align*}
  and the result follows.
\end{proof}

The following results are standard, but recorded here for convenience:
\begin{lemma}[Horn and Johnson \cite{HornJohnson2012_MatrixAnalysis}]\label{lem:posSemiDefBlocks}
  Let \(\b{H}=\left(
    \begin{smallmatrix}
      \b{A}&\b{B} \\
      \b{B}^{T}&\b{C}
    \end{smallmatrix}
  \right)\in \mathbb{R}^{(p+q)\times(p+q)}\) be symmetric, \(\b{A}\in\mathbb{R}^{p\times p}\), \(\b{B}\in\mathbb{R}^{p\times q}\), \(\b{C}\in\mathbb{R}^{q\times q}\).
  Then \(\b{H}\) is positive (semi-) definite if and only if \(\b{A}\) and \(\b{C}\) are positive (semi-) definite and there exists a contraction \(\b{X}\in \mathbb{R}^{p\times q}\) (i.e.~ all singular values are smaller or equal than 1) such that \(\b{B}=\b{A}^{\frac{1}{2}}\b{X}\b{C}^{\frac{1}{2}}\).
\end{lemma}

\begin{lemma}[Horn and Johnson \cite{HornJohnson1991_TopicsMatrixAnalysis}]\label{lem:semiDefHadamard}
  Let \(\b{Z}\), \(\b{Y}\) be symmetric positive (semi-) definite matrices. Then \(\b{Z} \circ \b{Y}\) is symmetric positive (semi-) definite.
\end{lemma}

We now prove the following estimates for the Hadamard power of matrices:
\begin{lemma}\label{lem:hadamardEst2}
  Let \(\b{Z}\in\mathbb{R}^{n\times m}\). Then
  \begin{align*}
    \norm{\b{Z}^{\circ2}}{2} & \leq \max_{i=1,\dots,m}\norm{\b{Z}\vec{e}_{m}^{\,(i)}}{2} \max_{j=1,\dots,n}\norm{\b{Z}^{T}\vec{e}_{n}^{\,(j)}}{2} \leq \norm{\b{Z}}{2}^{2},
  \end{align*}
  where \(\left\{\vec{e}_{m}^{\,(i)}\right\}\) and \(\left\{\vec{e}_{n}^{\,(j)}\right\}\) are the canonical unit basis vectors in \(\mathbb{R}^{m}\) and \(\mathbb{R}^{n}\) respectively.
\end{lemma}
\begin{proof}
  Both the matrix \(\left(\begin{array}{cc}\b{Z} & \b{I}_{n}\end{array}\right)^{T}\left(\begin{array}{cc}\b{Z}& \b{I}_{n}\end{array}\right)\) and \(\left(\begin{array}{cc}\b{I}_{m}& \b{Z}^{T}\end{array}\right)^{T}\left(\begin{array}{cc}\b{I}_{m}& \b{Z}^{T}\end{array}\right)\) are symmetric positive semi-definite.
  Using \Cref{lem:semiDefHadamard}, their Hadamard product
  \begin{align*}
    \left(
    \begin{array}{cc}
      \left(\b{Z}^{T}\b{Z}\right)\circ \b{I}_{m}& \left(\b{Z}^{T}\right)^{\circ 2} \\
      \b{Z}^{\circ 2} & \b{I}_{n}\circ \left(\b{Z}\b{Z}^{T}\right)
    \end{array}
                        \right)
  \end{align*}
  is also positive semi-definite.
  Hence, by \Cref{lem:posSemiDefBlocks},
  \begin{align*}
    \norm{\b{Z}^{\circ 2}}{2}^{2} \leq \norm{\left(\b{Z}\b{Z}^{T}\right)\circ \b{I}_{n}}{2} \norm{\left(\b{Z}^{T}\b{Z}\right)\circ \b{I}_{m}}{2}=\max_{i}\norm{\b{Z}\vec{e}_{m}^{\,(i)}}{2}^{2} \max_{i}\norm{\b{Z}^{T}\vec{e}_{n}^{\,(j)}}{2}^{2}.
  \end{align*}
  The second inequality follows trivially.
\end{proof}

\begin{definition}[Energy norms]
  For matrices \(\b{Z}\in\mathbb{R}^{n_{\ell}\times n_{\ell}}\), we define the usual matrix energy norm \(\norm{\b{Z}}{A}\) as well as the double energy norm \(\norm{\b{Z}}{A^{2}}\) to be \(\norm{\b{Z}}{A} = \norm{\b{A}_{\ell}^{\frac{1}{2}} \b{Z} \b{A}_{\ell}^{-\frac{1}{2}}}{2}\), and \(\norm{\b{Z}}{A^{2}} = \norm{\b{A}_{\ell} \b{Z} \b{A}_{\ell}^{-1}}{2}\).
  For matrices \(\b{W}\in\mathbb{R}^{n_{\ell}^{2}\times n_{\ell}^{2}}\), we define the tensor energy norm \(\norm{\b{W}}{A}\) and the tensor double energy norm \(\norm{\b{W}}{A^{2}}\) to be \(\norm{\b{W}}{A} = \norm{\left(\b{A}_{\ell}^{\frac{1}{2}}\right)^{\otimes2} \b{Z} \left(\b{A}_{\ell}^{-\frac{1}{2}}\right)^{\otimes2}}{2}\), and \(\norm{\b{W}}{A^{2}} = \norm{\b{A}_{\ell}^{\otimes2} \b{Z} \left(\b{A}_{\ell}^{-1}\right)^{\otimes2}}{2}\).
  In all cases \(\norm{\bullet}{2}\) is the spectral norm.
\end{definition}

The following Lemma permits us to expand the second moment of a fault-prone iteration matrix in terms of expectations and variances of the fault matrices:
\begin{lemma}\label{lem:prodPert}
  Let \(\bc{B}_{i}\), \(i=1,\dots,3\), be independent random matrices.
  Then
  \begin{align}
    \Exp \left[ \left(\b{I}-\bc{B}_{1}\right)^{\otimes2}\right] = \Exp\left[\b{I}-\bc{B}_{1}\right]^{\otimes2} + \Var\left[\bc{B}_{1}\right]
    \label{eq:oneRandomMatrixAgain}
  \end{align}
  and
  \begin{align}
    & \Exp \left[ \left(\b{I}-\bc{B}_{1}\bc{B}_{2}\bc{B}_{3}\right)^{\otimes2}\right] \label{eq:threeRandomMatrices}=\\
     & \Exp\left[\b{I}-\bc{B}_{1}\bc{B}_{2}\bc{B}_{3}\right]^{\otimes2}\nonumber \\
     +&  \Var\left[\bc{B}_{1}\right] \Exp\left[\bc{B}_{2}\bc{B}_{3}\right]^{\otimes2}
      +  \Exp\left[\bc{B}_{1}\right]^{\otimes2} \Var\left[\bc{B}_{2}\right] \Exp\left[\bc{B}_{3}\right]^{\otimes2}
      +  \Exp\left[\bc{B}_{1}\bc{B}_{2}\right]^{\otimes2} \Var\left[\bc{B}_{3}\right] \nonumber\\
     + &\Var\left[\bc{B}_{1}\right] \Var\left[\bc{B}_{2}\right] \Exp\left[\bc{B}_{3}\right]^{\otimes2}
      +  \Var\left[\bc{B}_{1}\right] \Exp\left[\bc{B}_{2}\right]^{\otimes2} \Var\left[\bc{B}_{3}\right]
      +  \Exp\left[\bc{B}_{1}\right]^{\otimes2} \Var\left[\bc{B}_{2}\right] \Var\left[\bc{B}_{3}\right] \nonumber\\
    + & \Var\left[\bc{B}_{1}\right] \Var\left[\bc{B}_{2}\right] \Var\left[\bc{B}_{3}\right]. \nonumber
  \end{align}
\end{lemma}
\begin{proof}
  \Cref{eq:oneRandomMatrixAgain} has already been shown in \cref{eq:oneRandomMatrix}.
  To obtain the identity~\cref{eq:threeRandomMatrices} we multiply out the tensor product, complete the square to recover the first term, and then use that
  \begin{align*}
    \Exp \left[ \bc{B}_{j}^{\otimes2} \right] = \Var \left[ \bc{B}_{j} \right] + \Exp \left[ \bc{B}_{j} \right]^{\otimes2}.
  \end{align*}
  The proof can easily be generalised to arbitrarily many random matrices~\cite{Glusa2017_MultigridDomainDecompositionMethods}.
\end{proof}

Since the proof on level \(L\) only involves the levels \(L\) and \(L-1\), we will drop the first subscript and replace the second one with a subscript \(C\) for the remainder of this work.

We set
\begin{align*}
  \Exp \left[ \bc{X}^{(\bullet)}\right] &= e^{(\bullet)}\b{I}, & \Var \left[ \bc{X}^{(\bullet)}\right] &= \b{V}^{(\bullet)}.
\end{align*}
Using \Cref{lem:prodPert} the second moment of the fault-prone coarse grid correction and smoothener can be written as
\begin{align}
  \Exp \left[ \left(\bc{E}^{CG}\right)^{\otimes2} \right]
  & = \Exp \left[ \bc{E}^{CG} \right]^{\otimes2}+\b{C}^{(P)}+\b{C}^{(R)}+\b{C}^{(\rho)} \label{eq:sumCG}\\
  &\quad+\b{C}^{(P,R)}+\b{C}^{(P,\rho)}+\b{C}^{(R,\rho)}+\b{C}^{(P,R,\rho)}, \nonumber\\
  \Exp \left[ \left(\bc{E}^{S}\right)^{\otimes2} \right] &= \Exp \left[ \bc{E}^{S} \right]^{\otimes2} + \b{C}^{(S)},\label{eq:sumSmoothener}
\end{align}
with
\begin{align*}
  \Exp \left[ \bc{E}^{CG} \right]   &=\b{E}^{CG}+ \left(1-e^{(P)}e^{(R)}e^{(\rho)}\right)\left(\b{I}-\b{E}^{CG}\right), \\
  \b{C}^{(P)}     & = \left(e^{(R)}e^{(\rho)}\right)^{2} \b{V}^{(P)}\left(\b{P}\b{A}_{C}^{-1}\b{R} \b{A}\right)^{\otimes2}, \\
  \b{C}^{(R)}     & = \left(e^{(P)}e^{(\rho)}\right)^{2} \left(\b{P}\b{A}_{C}^{-1}\right)^{\otimes2}\b{V}^{(R)}\left(\b{R} \b{A}\right)^{\otimes2}, \\
  \b{C}^{(\rho)}     & = \left(e^{(P)}e^{(R)}\right)^{2} \left(\b{P}\b{A}_{C}^{-1}\b{R}\right)^{\otimes2} \b{V}^{(\rho)} \b{A}^{\otimes2}, \\
  \b{C}^{(P,R)}   & = \left(e^{(\rho)}\right)^{2} \b{V}^{(P)} \left(\b{P}\b{A}_{C}^{-1}\right)^{\otimes2}\b{V}^{(R)}\left(\b{R} \b{A}^{\otimes2}\right)^{\otimes2}, \\
  \b{C}^{(P,\rho)}   & = \left(e^{(R)}\right)^{2} \b{V}^{(P)} \left(\b{P}\b{A}_{C}^{-1}\b{R}\right)^{\otimes2} \b{V}^{(\rho)} \b{A}^{\otimes2}, \\
  \b{C}^{(R,\rho)}   & = \left(e^{(P)}\right)^{2} \left(\b{P}\b{A}_{C}^{-1}\right)^{\otimes2} \b{V}^{(R)} \b{R}^{\otimes2} \b{V}^{(\rho)}\b{A}^{\otimes2}, \\
  \b{C}^{(P,R,\rho)} & = \b{V}^{(P)} \left(\b{P}\b{A}_{C}^{-1}\right)^{\otimes2}\b{V}^{(R)}\b{R}^{\otimes2} \b{V}^{(\rho)} \b{A}^{\otimes2}, \\
  \Exp \left[ \bc{E}^{S} \right]    &= \b{E}^{S} + e^{(S)} \left(\b{I}-\b{E}^{S}\right), \\
  \b{C}^{(S)}     & =\b{V}^{(S)}\left(\b{N}\b{A}\right)^{\otimes2}.
\end{align*}
When the prolongation is protected, or not subject to faults, we have \(e^{(P)}=1\) and \(\b{V}^{(P)}=\b{0}\), so that all \(\b{C}^{(\bullet)}\) with a superscript containing \(P\) are zero.

\section{Proof of \texorpdfstring{\Cref{thm:perturbedTGconv}}{Theorem 3}}
\label{sec:proof-TGconv}

\begin{proof}
  We have that
  \begin{align*}
    e^{(\bullet)} & = 1-q, & \b{V}^{(P)} & =\b{V}^{(\rho)} = \b{V}^{(S)}= q(1-q)\b{K}, & \b{V}^{(R)} & = q(1-q)\b{K_{C}},
  \end{align*}
  with
  \begin{align*}
    \b{K} & = \blockdiag\left(\vec{e}^{\,(i)} \otimes \left(\vec{e}^{\,(i)}\right)^{T},~i=1,\ldots,n\right), \\
    \b{K}_{C} & = \blockdiag\left(\vec{e}_{C}^{\,(i)} \otimes \left(\vec{e}_{C}^{\,(i)}\right)^{T},~i=1,\ldots,n_{C}\right).
  \end{align*}
  Here, \(\vec{e}^{\,(i)}\) and  \(\vec{e}_{C}^{\,(i)}\) are the canonical unit basis vectors of \(\mathbb{R}^{n}\) and \(\mathbb{R}^{n_{C}}\) respectively.
  Adding and subtracting \(\Exp \left[ \bc{E}^{TG} \right]^{\otimes2}\) from \(\Exp \left[ \left(\bc{E}^{TG}\right)^{\otimes2} \right]\), we estimate using the energy norm on the tensor space
  \begin{align*}
    \rho\left(\Exp \left[ \left(\bc{E}^{TG}\right)^{\otimes2} \right]\right)
    & \leq \norm{\Exp \left[ \left(\bc{E}^{TG}\right)^{\otimes2} \right]}{A} \\
    &\leq \norm{\Exp \left[ \bc{E}^{TG} \right]}{A}^{2}
      + \norm{\Exp \left[ \left(\left(\bc{E}^{S,\text{post}}\right)^{\nu_{2}}\bc{E}^{CG}\left(\bc{E}^{S,\text{pre}}\right)^{\nu_{1}}\right)^{\otimes2} \right]}{A} \\
    &\quad- \norm{\Exp \left[ \left(\bc{E}^{S,\text{post}}\right)^{\nu_{2}}\bc{E}^{CG}\left(\bc{E}^{S,\text{pre}}\right)^{\nu_{1}} \right]}{A}^{2}.
  \end{align*}
  We then use the subadditivity of \(\norm{\bullet}{A}\) and equations \cref{eq:sumCG} and \cref{eq:sumSmoothener} to write
  \begin{align}
    \rho\left(\Exp
    \left[ \left(\bc{E}^{TG}\right)^{\otimes2} \right]\right) & \leq
    \begin{multlined}[t]
      \norm{\Exp \left[ \bc{E}^{TG} \right]}{A}^{2} \nonumber\\
      \quad + \Big(\norm{\Exp \left[ \bc{E}^{CG} \right]}{A}^{2}+\norm{\b{C}^{(P)}}{A} + \norm{\b{C}^{(R)}}{A} + \norm{\b{C}^{(\rho)}}{A} \nonumber\\
      + \norm{\b{C}^{(P,R)}}{A} + \norm{\b{C}^{(P,\rho)}}{A} + \norm{\b{C}^{(R,\rho)}}{A} + \norm{\b{C}^{(P,R,\rho)}}{A}\Big) \nonumber
    \end{multlined} \nonumber\\
                                                              &\qquad \times\left(\norm{\Exp \left[ \bc{E}^{S} \right]}{A}^{2}+\norm{\b{C}^{(S)}}{A}\right)^{\nu_{1}+\nu_{2}} \nonumber\\
                                                              &\quad - \norm{\Exp \left[ \bc{E}^{CG} \right]}{A}^{2} \norm{\Exp \left[ \bc{E}^{S} \right]}{A}^{2\left(\nu_{1}+\nu_{2}\right)}. \label{eq:decomposition}
  \end{align}
  First, we estimate the terms involving only first moments of the fault matrices.
  Using that \(\b{E}^{CG}\) is an \(A\)-orthogonal projection and that the damped Jacobi smoothener is convergent, we find
  \begin{align*}
    \norm{\Exp \left[ \bc{E}^{CG} \right]}{A} & \leq \norm{\b{E}^{CG}}{A} + \left(1-\left(1-q\right)^{3}\right) \norm{\b{I}-\b{E}^{CG}}{A} \leq 1 + Cq,\\
    \norm{\Exp \left[ \bc{E}^{S} \right]}{A} & \leq \norm{\b{E}^{S}}{A} + q \norm{\b{I}-\b{E}^{S}}{A} \leq 1 + Cq,
  \end{align*}
  and therefore
  \begin{align*}
    \norm{\Exp \left[ \bc{E}^{TG} \right]}{A} & \leq \norm{\b{E}^{TG}}{A} +\norm{\Exp \left[ \bc{E}^{S} \right]}{A}^{\nu_{1}+\nu_{2}} \norm{\Exp \left[ \bc{E}^{CG} \right]}{A} - \norm{\b{E}^{S}}{A}^{\nu_{1}+\nu_{2}} \norm{\b{E}^{CG}}{A}\\
                                              &= \norm{\b{E}^{TG}}{A} + Cq.
  \end{align*}
  Next, we estimate all the terms \(\b{C}^{(\bullet)}\) involving variances of the fault matrices.
  We bound
  \begin{align*}
    \frac{1}{q}\norm{\b{C}^{(\rho)}}{A} & \leq \rho
                                          \left[ \left(\b{A}^{\frac{1}{2}}\b{P}\b{A}_{C}^{-1}\b{R}\right)^{\otimes2}\b{K}\b{A}^{\otimes2}\b{K}\left(\b{P}\b{A}_{C}^{-1}\b{R}\b{A}^{\frac{1}{2}}\right)^{\otimes2} \right]^{\frac{1}{2}} \\
                                        & = \rho
                                          \left[ \left(\b{P}\b{A}_{C}^{-1}\b{R}\b{A}\b{P}\b{A}_{C}^{-1}\b{R}\right)^{\otimes2}\b{K}\b{A}^{\otimes2}\b{K}\right]^{\frac{1}{2}} \\
                                        & =\rho
                                          \left[ \left(\b{P}\b{A}_{C}^{-1}\b{R}\right)^{\circ2}\b{A}^{\circ2}\right]^{\frac{1}{2}} \\
                                        & \leq \norm{\left(\b{P}\b{A}_{C}^{-1}\b{R}\right)^{\circ2}}{2}^{\frac{1}{2}}\norm{\b{A}^{\circ2}}{2}^{\frac{1}{2}}.
  \end{align*}
  Here, we used that \(\b{K}=\b{K}^{2}\) and that for compatible \(\b{Z}\)
  \begin{align*}
    \left(\b{K}\b{Z}^{\otimes2}\b{K}\right)_{in+p,jn+q} &= \b{K}_{in+p,in+p}\b{Z}_{ij}\b{Z}_{pq}\b{K}_{jn+q,jn+q}
                                                        =
                                                          \begin{cases}
                                                            \left(\b{Z}^{\circ 2}\right)_{ij} & \text{if } i=p,~j=q, \\
                                                            0 & \text{else}
                                                          \end{cases}
  \end{align*}
  for \(1\leq i,j,p,q\leq n\).
  Similarly, we find
  \begin{align*}
    \frac{1}{q}\norm{\b{C}^{(P)}}{A} & \leq \norm{\b{A}^{\circ2}}{2}^{\frac{1}{2}} \norm{\left(\b{P}\b{A}_{C}^{-1}\b{R}\right)^{\circ2}}{2}^{\frac{1}{2}}, \\
    \frac{1}{q}\norm{\b{C}^{(R)}}{A} & \leq \norm{\b{A}_{C}^{\circ2}}{2}^{\frac{1}{2}} \norm{\left(\b{A}_{C}^{-1}\right)^{\circ2}}{2}^{\frac{1}{2}}, \\
    \frac{1}{q^{2}}\norm{\b{C}^{(P,R)}}{A} & \leq \norm{\b{A}^{\circ2}}{2}^{\frac{1}{2}} \norm{\b{A}_{C}^{\circ2}}{2}^{\frac{1}{2}} \norm{\left(\b{A}_{C}^{-1}\b{R}\right)^{\circ2}}{2}, \\
    \frac{1}{q^{2}}\norm{\b{C}^{(P,\rho)}}{A} & \leq \norm{\b{A}^{\circ2}}{2} \norm{\left(\b{P}\b{A}_{C}^{-1}\b{R}\right)^{\circ2}}{2}, \\
    \frac{1}{q^{2}}\norm{\b{C}^{(R,\rho)}}{A} & \leq \norm{\b{P}^{\circ 2}}{2} \norm{\b{R}^{\circ 2}}{2} \norm{\b{A}^{\circ2}}{2}^{\frac{1}{2}} \norm{\left(\b{A}_{C}^{-1}\right)^{\circ2}}{2}^{\frac{1}{2}}, \\
    \frac{1}{q^{3}}\norm{\b{C}^{(P,R,\rho)}}{A} & \leq \norm{\b{P}^{\circ 2}}{2}^{2} \norm{\b{A}^{\circ2}}{2} \norm{\left(\b{A}_{C}^{-1}\b{R}\right)^{\circ2}}{2}, \\
    \frac{1}{q}\norm{\b{C}^{(S)}}{A} &\leq \norm{\left(\b{N}\b{A}\right)^{\circ2}}{2} .
  \end{align*}
  In the last inequality, we used that \(\b{V}^{(S)}\) and \(\b{N}^{\otimes2}\) commute.
  Using \Cref{lem:hadamardEst2,as:approximation,as:prolongation} we find
  \begin{align*}
    \norm{\b{P}^{\circ2}}{2} &\leq \norm{\b{P}}{2}^{2} \leq C, \\
    \norm{\b{A}^{\circ2}}{2} &\leq \norm{\b{A}}{2}^{2}, \\
    \norm{\b{A_{C}}^{\circ2}}{2} &\leq \norm{\b{A_{C}}}{2}^{2} \leq \norm{\b{A}}{2}^{2}, \\
    \norm{\left(\b{P}\b{A}_{C}^{-1}\b{R}\right)^{\circ2}}{2} & \leq \max_{i}\norm{\b{P}\b{A}_{C}^{-1}\b{R}\vec{e}^{\,(i)}}{2}^{2} \\
                             & \leq \norm{\b{P}\b{A}_{C}^{-1}\b{R}\b{A}}{2}^{2}\max_{i}\norm{\b{A}^{-1}\vec{e}^{\,(i)}}{2}^{2} \\
                             & \leq C \max_{i}\norm{\b{A}^{-1}\vec{e}^{\,(i)}}{2}^{2}, \\
    \norm{\left(\b{A}_{C}^{-1}\b{R}\right)^{\circ2}}{2}&\leq \max_{i}\norm{\b{P}\b{A}_{C}^{-1}\vec{e}_{C}^{\,(i)}}{2} \max_{j}\norm{\b{A}_{C}^{-1}\b{R}\vec{e}^{\,(j)}}{2} \\
                             &\leq C\max_{i}\norm{\b{A}_{C}^{-1}\vec{e}_{C}^{\,(i)}}{2} \norm{\b{A}_{C}^{-1}\b{R}\b{A}}{2}\max_{j}\norm{\b{A}^{-1}\vec{e}^{\,(j)}}{2}\\
                             &\leq C \max_{i}\norm{\b{A}_{C}^{-1}\vec{e}_{C}^{\,(i)}}{2}\max_{j}\norm{\b{A}^{-1}\vec{e}^{\,(j)}}{2}, \\
    \norm{\left(\b{N}\b{A}\right)^{\circ2}}{2}&\leq  \norm{\b{N}\b{A}}{2}^{2} \leq C.
  \end{align*}
  Now, because \(\b{A}\) is the finite element discretization of a second order PDE over a quasi-uniform mesh, we obtain (see Theorem 9.11 in \cite{ErnGuermond2004_TheoryPracticeFiniteElements}),
  \begin{align*}
    \norm{\b{A}}{2} & \leq C h^{d-2}.
  \end{align*}
  Since
  \begin{align*}
    \norm{\b{A}\vec{e}^{\,(i)}}{2} &\leq C h^{-\frac{d}{2}}\norm{u_{L}}{L^{2}},
  \end{align*}
  where \(u_{L}\in V_{L}\) is given by
  \begin{align*}
    a\left(u_{L},v\right)=v\left(\vec{x}_{i}\right), \quad \forall v\in V_{L},
  \end{align*}
  we can apply Lemma~\ref{lem:Ainvbound} to obtain bounds for \(\max_{j}\norm{\b{A}^{-1}\vec{e}^{\,(j)}}{2}\) (and equally for \(\max_{i}\norm{\b{A}_{C}^{-1}\vec{e}_{C}^{\,(i)}}{2}\)).
  \begin{align*}
    \frac{1}{q}\norm{\b{C}^{(P)}}{A},~\frac{1}{q}\norm{\b{C}^{(R)}}{A},~\frac{1}{q}\norm{\b{C}^{(\rho)}}{A},~ \frac{1}{q^{2}}\norm{\b{C}^{(R,\rho)}}{A} &\leq C \begin{cases}
      h^{\frac{d-4}{2}} & d<4, \\
      \sqrt{1+ \left|\log h\right|} & d=4, \\
      1 & d> 4,
    \end{cases}\\
    \frac{1}{q^{2}}\norm{\b{C}^{(P,R)}}{A},~\frac{1}{q^{2}}\norm{\b{C}^{(P,\rho)}}{A}, ~ \frac{1}{q^{3}}\norm{\b{C}^{(P,R,\rho)}}{A} & \leq C \begin{cases}
      h^{d-4} & d<4, \\
      \left(1+ \left|\log h\right|\right) & d=4, \\
      1 & d> 4,
    \end{cases}\\
    \frac{1}{q}\norm{\b{C}^{(S)}}{A} &\leq C.
  \end{align*}
  Therefore, we obtain from \cref{eq:decomposition}

  \begin{align*}
    \rho\left(\Exp \left[ \left(\bc{E}^{TG}\right)^{\otimes2} \right] \right) &\leq \norm{\b{E}^{TG}}{A}^{2} + C\begin{cases}
      q^{2}h^{d-4} & d<4, \\
      q^{2}\left(1+\left|\log h\right|\right) & d=4, \\
      q^{2} & d> 4,
    \end{cases}
  \end{align*}
  and hence
  \begin{align*}
    \varrho\left(\bc{E}^{TG}\right) & \leq \norm{\b{E}^{TG}}{A} +
                                      C\begin{cases}
                                        qn^{\frac{4-d}{2d}} & d < 4, \\
                                        q \sqrt{\log n} & d=4, \\
                                        q & d > 4,
                                      \end{cases}
  \end{align*}
  where we used that \(n\approx h^{-d}\).
\end{proof}

In order to conclude, we need the following technical estimate:
\begin{lemma}\label{lem:Ainvbound}
  Let \(\Omega\in C^{2}\) or \(\Omega\) a convex polyhedron and let \(u_{L}\in V_{L}\) be the unique solution of
  \begin{align*}
    a\left(u_{L},v\right)=v\left(\vec{x}_{i}\right), \quad \forall v\in V_{L}.
  \end{align*}
  Then
  \begin{align*}
    \norm{u_{L}}{L^{2}}\leq
    C\begin{cases}
      1 & d < 4, \\
      \left(1+\left|\log h\right|\right)^{\frac{1}{2}} & d=4, \\
      h^{2-\frac{d}{2}} & d> 4,
    \end{cases}
  \end{align*}
  where \(C\) is a constant independent of \(h\).

\end{lemma}
\begin{proof}
  Let \(f\in V_{L}\) be the unique function that corresponds to the load and write \(f=\vec{F}\cdot \phi\) with \(\vec{F}\) its coefficient vector and \(\phi\) the vector of shape functions.
  Then
  \begin{align*}
    \vec{e}^{\,(i)} & =\left(\phi,f\right)_{L^{2}} = \left(\phi,\phi\right)_{L^{2}}\cdot \vec{F}  = \b{M}\vec{F}.
  \end{align*}
  Since by Theorem 9.8 in \cite{ErnGuermond2004_TheoryPracticeFiniteElements}
  \begin{align}
    C h^{d}\b{I} \leq \b{M} \leq C h^{d}\b{I},\label{eq:mass}
  \end{align}
  we have
  \begin{align}
    \norm{f}{H^{m}}\leq C h^{-d}\norm{\phi^{(i)}}{H^{m}}. \label{eq:20}
  \end{align}
  Now \(u_{L}\) is an approximation to \(u\in V\) that solves
  \begin{align*}
    a(u,v)= \left(f, v\right)_{L^{2}}, \quad \forall v\in V.
  \end{align*}
  Since \(f\in V\subset L^{2}\left(\Omega\right)\), we find by the Aubin-Nitsche Lemma \cite{ErnGuermond2004_TheoryPracticeFiniteElements} that
  \begin{align*}
    \norm{u_{L}-u}{L^{2}}\leq Ch^{2}\norm{f}{L^{2}}.
  \end{align*}
  Consider the solution \(u^{*}\in V\) to the dual problem
  \begin{align*}
    a(v,u^{*})= \left(u,v\right)_{L^{2}}, \quad \forall v\in V.
  \end{align*}
  By elliptic regularity \cite{ErnGuermond2004_TheoryPracticeFiniteElements}, we have
  \begin{align*}
    \norm{u^{*}}{H^{2}}\leq C\norm{u}{L^{2}}.
  \end{align*}
  Moreover,
  \begin{align*}
    \left(u,u\right)_{L^{2}} = a(u,u^{*}) = \left(f,u^{*}\right)_{L^{2}},
  \end{align*}
  so
  \begin{align*}
    \norm{u}{L^{2}}^{2} & \leq \norm{f}{H^{-2}}\norm{u^{*}}{H^{2}} \leq C\norm{f}{H^{-2}}\norm{u}{L^{2}},
  \end{align*}
  and hence \(\norm{u}{L^{2}}\leq C\norm{f}{H^{-2}}\).
  Therefore, by triangle inequality and \cref{eq:20}, we find
  \begin{align*}
    \norm{u_{L}}{L^{2}} &\leq Ch^{2}\norm{f}{L^{2}} + C\norm{f}{H^{-2}} \leq Ch^{-d}\left(h^{2}\norm{\phi^{(i)}}{L^{2}} + \norm{\phi^{(i)}}{H^{-2}}\right).
  \end{align*}
  Applying the estimates for \(\norm{\phi^{(i)}}{H^{m}}\) from Theorem 4.8 in \cite{AinsworthMcleanEtAl1999_ConditioningBoundaryElementEquations}, we obtain
  \begin{align*}
    \norm{u_{L}}{L^{2}}\leq
    C\begin{cases}
      1 & d < 4, \\
      \left(1+\left|\log h\right|\right)^{\frac{1}{2}} & d=4, \\
      h^{2-\frac{d}{2}} & d> 4.
    \end{cases}
  \end{align*}
\end{proof}

\section{Proof of \texorpdfstring{\Cref{thm:TGNoProlong}}{Theorem 4}}
\label{sec:proof-TGNoProlong}

\begin{proof}
  Adding and subtracting \(\Exp \left[ \bc{E}^{TG}\left(\nu_{1},\nu_{2}\right) \right]^{\otimes2}\) from \(\Exp \left[ \left(\bc{E}^{TG}\left(\nu_{1},\nu_{2}\right)\right)^{\otimes2} \right]\), we estimate in \(\norm{\bullet}{A^{2}}\)
  \begin{align*}
    &\rho\left(\Exp \left[ \left(\bc{E}^{TG}\left(\nu_{1},\nu_{2}\right)\right)^{\otimes2} \right]\right)
    \leq \norm{\Exp \left[ \left(\bc{E}^{TG}\left(\nu_{1},\nu_{2}\right)\right)^{\otimes2} \right]}{A^{2}} \\
    \leq&
    \begin{multlined}[t]
      \norm{\Exp \left[ \bc{E}^{TG}\left(\nu_{1},\nu_{2}\right) \right]}{A^{2}}^{2} + \left(\norm{\Exp \left[ \bc{E}^{CG} \right]}{A^{2}}^{2} + \norm{\b{C}^{(R)}}{A^{2}} + \norm{\b{C}^{(\rho)}}{A^{2}}
      +\norm{\b{C}^{(R,\rho)}}{A^{2}}\right)\\
     \times\left(\norm{\Exp \left[ \bc{E}^{S} \right]}{A^{2}}^{2}+\norm{\b{C}^{(S)}}{A^{2}}\right)^{\nu_{1}+\nu_{2}} \\
     - \norm{\Exp \left[ \bc{E}^{CG} \right]}{A^{2}}^{2} \norm{\Exp \left[ \bc{E}^{S} \right]}{A^{2}}^{2\left(\nu_{1}+\nu_{2}\right)}.
    \end{multlined}
  \end{align*}
  We then get by \Cref{as:faults,as:approximation,as:prolongation} and \(\left(e^{\bullet}\right)^{2}\leq 1+2C\varepsilon + C^{2}\varepsilon^{2}\leq 1 + C\varepsilon\) that
  \begin{align*}
    \norm{\b{C}^{(R)}}{A^{2}} & \leq \varepsilon \left(e^{(R)}\right)^{2}\norm{\b{A}_{C}^{-1}\b{R}\b{A}}{2}^{2}\norm{\b{R}}{2}^{2}\leq \varepsilon \left(e^{(R)}\right)^{2}\underline{C}_{p}^{2}\overline{C}_{p}^{2}C_{A}^{2}  \leq C\varepsilon, \\
    \norm{\b{C}^{(\rho)}}{A^{2}} & \leq \varepsilon \left(e^{(\rho)}\right)^{2}\norm{\b{P}\b{A}_{C}^{-1}\b{R}\b{A}}{2}^{2} \leq \varepsilon \left(e^{(\rho)}\right)^{2}C_{A}^{2} \leq C\varepsilon, \\
    \norm{\b{C}^{(R,\rho)}}{A^{2}} & \leq \varepsilon^{2} \norm{\b{A}_{C}^{-1}\b{R}\b{A}}{2}^{2}\norm{\b{R}}{2}^{2} \leq \varepsilon^{2}\underline{C}_{p}^{2}\overline{C}_{p}^{2}C_{A}^{2}  \leq C\varepsilon^{2}.
  \end{align*}
  We also estimate
  \begin{align*}
    \norm{\b{C}^{(S)}}{A^{2}} & = \norm{\b{A}^{\otimes2}\b{V}^{(S)}\b{N}^{\otimes2}}{2} = \norm{\left(\b{A}\b{N}\right)^{\otimes2}\b{V}^{(S)}}{2} \leq \varepsilon\norm{\b{A}\b{N}}{2}^{2} \leq C\varepsilon.
  \end{align*}
  Here, we used \Cref{as:faults,as:smoother} and that \(\b{N}\) and \(\bc{X}^{(S)}\) commute.
  Moreover,
  \begin{align*}
    \norm{\Exp \left[ \bc{E}^{CG} \right]}{A^{2}} & =\norm{\Exp \left[ \bc{E}^{CG} \right]^{T}}{2} = \norm{\Exp \left[ \bc{E}^{CG} \right]}{2} \\
                                                  & \leq \norm{\b{E}^{CG}}{2} +  \left| 1-e^{(R)}e^{(\rho)} \right| \norm{\b{I}- \b{E}^{CG} }{2} \\
                                                  & \leq C_{A}\left(1+ \left| 1-e^{(R)}e^{(\rho)} \right|\right) \leq C\left(1+\varepsilon\right), \\
    \norm{\Exp \left[ \bc{E}^{S} \right]}{A^{2}} & = \norm{\Exp \left[ \bc{E}^{S} \right]^{T}}{2}= \norm{\Exp \left[ \bc{E}^{S} \right]}{2} \\
                                                  & \leq \norm{\b{E}^{S}}{2} +  \left| 1-e^{(S)} \right| \norm{\b{I}- \b{E}^{S} }{2} \leq C(1+\varepsilon), \\
    \norm{\Exp \left[ \bc{E}^{TG}\left(\nu_{1},\nu_{2}\right) \right]}{A^{2}} & = \norm{\Exp \left[ \bc{E}^{TG}\left(\nu_{1},\nu_{2}\right) \right]^{T}}{2}= \norm{\Exp \left[ \bc{E}^{TG}\left(\nu_{2},\nu_{1}\right) \right]}{2} \\
                                                  & = \norm{\b{E}^{TG}\left(\nu_{2},\nu_{1}\right) }{2} + \norm{\Exp \left[ \bc{E}^{TG}\left(\nu_{2},\nu_{1}\right) \right]}{2} - \norm{\b{E}^{TG}\left(\nu_{2},\nu_{1}\right) }{2}\\
                                                  &\leq \norm{ \b{E}^{TG}\left(\nu_{2},\nu_{1}\right) }{2} + C\varepsilon.
  \end{align*}
  Here, we used that by \Cref{as:faults}
  \begin{align*}
    \left| 1-e^{(\bullet)} \right|  & \leq C\varepsilon, &
    \left(e^{(\bullet)}\right)^{2}      & \leq 1 + C\varepsilon, &
    \left| 1-e^{(R)}e^{(\rho)} \right| & \leq C\varepsilon.
  \end{align*}
  Collecting all the terms, we have
  \begin{align*}
    \rho\left(\Exp\left[\left(\bc{E}^{TG}\left(\nu_{1},\nu_{2}\right)\right)^{\otimes2}\right]\right)
    & \leq \norm{\b{E}^{TG}\left(\nu_{2},\nu_{1}\right)}{2}^{2} + C\varepsilon.
  \end{align*}
  so that we finally obtain
  \begin{align*}
    \varrho\left(\bc{E}^{TG}\left(\nu_{1},\nu_{2}\right)\right) & \leq \norm{\b{E}^{TG}\left(\nu_{2},\nu_{1}\right)}{2} + C\varepsilon
  \end{align*}

  We conclude by observing that the Lyapunov spectral radius, just as the ordinary spectral radius, is invariant with respect to cyclic permutations.
\end{proof}

\section*{Acknowledgements}
This research used resources of the Oak Ridge Leadership Computing Facility at the Oak Ridge National Laboratory, which is supported by the Office of Science of the U.S. Department of Energy under Contract No. DE-AC05-00OR22725.

\bibliography{/home/cag/org/papers}
\bibliographystyle{siamplain}
\end{document}